\documentclass[12pt,a4paper,dk,reqno]{amsart}
\let\amsmarkboth\markboth
\usepackage{amsmath,amsfonts,amssymb,amsthm,amscd,dsfont}
\usepackage[english]{babel}
\usepackage[mathscr]{eucal}
\usepackage{graphicx}
\usepackage{appendix}

\makeatletter
\let\markboth\amsmarkboth
\bbl@redefine#1#2{%
   \def\bbl@arg{#1}%
   \ifx\bbl@arg\@empty
     \toks@{}%
   \else
     \toks@{\noexpand\foreignlanguage{%
              \languagename}{%
              \noexpand\bbl@restore@actives#1}}%
   \fi
   \def\bbl@arg{#2}%
   \ifx\bbl@arg\@empty
     \toks8{}%
   \else
     \toks8{\noexpand\foreignlanguage{%
              \languagename}{%
              \noexpand\bbl@restore@actives#2}}%
   \fi
   \edef\bbl@tempa{\the\toks@}%
   \edef\bbl@tempb{\the\toks8}%
   \protected@edef\bbl@tempa{%
     \noexpand\org@markboth{\bbl@tempa}{\bbl@tempb}}%
   \bbl@tempa
}

\DeclareRobustCommand*\ams@disablelinebreak{\def\\{ \ignorespaces}}
\def\maketitle{\par
   \@topnum\z@ %
   \@setcopyright
   \thispagestyle{firstpage}%
   \uppercasenonmath\shorttitle
   \ifx\@empty\shortauthors \let\shortauthors\shorttitle
   \else \andify\shortauthors
   \fi
   \@maketitle@hook
   \begingroup
   \@maketitle
   \@tempa
   \endgroup
   \c@footnote\z@
   \@cleartopmattertags
}

\makeatother

\numberwithin{equation}{section}

\newtheorem{thm}{Theorem}

\newtheorem{prop}{Proposition}
\newtheorem{lemme}{Lemma}
\newtheorem{defi}{Definition}
\newtheorem{coro}{Corollary}

\newcommand{\dsp}{\displaystyle}
\newcommand{\eps}{\varepsilon}

\newcommand{\R}{\mathbb{R}^2}
\newcommand{\dt}{\partial_t}
\newcommand{\loc}{\mathrm{loc}}

\newcommand{\dun}{\partial_1}
\newcommand{\ddeux}{\partial_2}
\newcommand{\E}{E_\eps(u)}

\newcommand{\er}{\mathcal{E}_{\eps,[U_d]}}
\newcommand{\ssum}{\sum_{i=1}^l}
\newcommand{\ai}{a_i}
\newcommand{\aio}{a_i^0}

\newcommand{\ue}{u_{\varepsilon}}
\newcommand{\ueo}{u_{\varepsilon}^0}
\newcommand{\uest}{u_{\varepsilon}^*}
\newcommand{\ust}{u^*}

\newcommand{\keps}{\frac{\delta}{|\log \eps|}}
\newcommand{\intr}{\int_{\R}}
\newcommand{\leps}{|\log \eps|}

\newcommand{\CGL}{$\mathrm{(CGL)}_\eps$ }

\begin{document}

\title{Dynamics of vortices for the Complex Ginzburg-Landau equation}
\author{Evelyne Miot}

\address[E. Miot]{Laboratoire J.-L. Lions UMR 7598,
Universit\'e Pierre et Marie Curie, 175 rue du Chevaleret, 75013
Paris, France} \email{miot@ann.jussieu.fr}

\begin{abstract}
    We study a complex Ginzburg-Landau equation in the plane, which has the
    form of a Gross-Pitaevskii equation with some dissipation added.  We
    focus on the regime corresponding to well-prepared unitary
    vortices and derive their asymptotic motion law.

    \medskip

    \noindent\tiny{\emph{2000 Mathematics Subject Classification:}
    35B20,35B40,35Q40,82D55. \\
\emph{Keywords:} Complex Ginzburg-Landau equation, Vortex dynamics.}
\end{abstract}

\date{September 29, 2008}
\maketitle


\section{Introduction}

\renewcommand{\theequation}{\arabic{equation}}


In this paper, we study the dynamics of vortices for a complex
Ginzburg-Landau equation on the plane, namely
\begin{equation*}
 \qquad \keps \dt \ue + \alpha i\dt \ue=\Delta \ue
+\frac{1}{\eps^2}\ue (1-|\ue|^2) \qquad \qquad
\mathrm{(\mathrm{CGL})_\eps}
\end{equation*}
where $\dsp \ue:\mathbb{R}_+\times \R \rightarrow \R$ is a complex
valued map. Here $\delta$, $\alpha$ and $\eps$ denote positive real
parameters, and we will mainly focus on the asymptotics as $\eps$
tends to zero while $\delta$ and $\alpha$ are kept fixed. Up to a
change of scale, we may further assume that $\alpha=1$, and we set
$k_\eps=\frac{\delta}{\leps}$.
 The complex Ginzburg-Landau equation \CGL reduces to the
Gross-Pitaevskii equation when $\delta=0$ and to the parabolic
Ginzburg-Landau equation when $\alpha=0$. Both the Gross-Pitaevskii
and the Ginzburg-Landau equations have been widely investigated in
the regime which we will consider (see e.g. \cite{CJ,LX,JS2,BJS} for
the Gross-Pitaevskii equation and \cite{JSon98,serf,BOS2} and
references therein for the parabolic Ginzburg-Landau equation).
Typical functions $u_\eps$ in this regime are given explicitely by
\begin{equation*}
\uest(a_i,d_i):=\prod_{i=1}^l u_{\eps,d_i}(z-a_i)=\prod_{i=1}^l
f_{1,d_i}\left(\frac{|z-a_i|}{\eps}\right)
\left(\frac{z-a_i}{|z-a_i|}\right)^{d_i},
\end{equation*}
where the points $a_i \in \R$, $d_i=\pm 1,$ and the functions
$f_{1,d_i}:\mathbb{R}^+ \mapsto [0,1]$ which satisfy
$f_{1,d_i}(0)=0, f_{1,d_i}(+\infty)=1$ are in some sense optimal
profiles. The points $a_i$ are called the vortices of the fields
$u_\eps$ and the $d_i$ their degrees. This class of functions
$u_\eps$ is of course not invariant by any of the flows
corresponding to these equations, but not far from it\footnote{see
the notion of well-preparedness in Section \ref{sect:resu}}, and it
is in particular possible to define notions of point vortices for
solutions of \CGL, at least in an asymptotic way as $\eps \to 0$,
and to study their dynamics. This dynamics is eventually governed by
a system of ordinary differential equations, at least before
collisions.

Two relevant quantities in the study of vortex dynamics are the
Ginzburg-Landau energy
\begin{equation*}
\E=\int_{\R}e_\eps(u)\,dx=\int_{\R}\frac{|\nabla
u|^2}{2}+\frac{(1-|u|^2)^2}{4\eps^2} \,dx,
\end{equation*}
through its energy density $e_\eps(u)$, and the Jacobian
\begin{equation*}
Ju=\frac{1}{2}\textrm{curl}(u\times \nabla u)
\end{equation*}
through its primitive $j(u) = u\times \nabla u.$ In the regime which
we will consider, one has
$$
\frac{e_\eps(u_\eps)}{\leps}\, dx \rightharpoonup \pi \sum_{i=1}^l
\delta_{a_i} \quad\text{and}\quad J u_\eps\, dx \rightharpoonup \pi
\sum_{i=1}^l d_i \delta_{a_i}
$$
as $\eps \to 0,$ which describes asymptotically the positions and
the degrees of the vortices. The quantity $e_\eps(u_\eps)$ was
especially used in the study of the parabolic Ginzburg-Landau
equation while $j(u_\eps)$ was used in the study of the
Gross-Pitaevskii equation. Here, we will rely on both of them.

In the case of the domain being the entire plane $\R$, which we
consider here, the reference fields $u_\eps(a_i,d_i)$ have infinite
Ginzburg-Landau energy $E_\eps$  whenever $d= \sum d_i \neq 0.$ In
\cite{BS}, a notion of renormalized energy\footnote{not to be merged
with the notion in \cite{BBH}.} for such data was introduced in
order to solve the Cauchy problem for the Gross-Pitaevskii equation.
This notion was later used in \cite{BJS} in order to study the
dynamics of vortices for the Gross-Pitaevskii equation in the plane.
Our definition of well-prepared data below and part of the
subsequent analysis is borrowed from \cite{BJS}.

\medskip

The complex Ginzburg-Landau equation \CGL, either in the plane or in
the real line, has been vastly considered in the literature,
especially as a model for amplitude oscillation in weakly nonlinear
systems undergoing a Hopf bifurcation (see e.g. \cite{ArKr} for a
survey paper). The mathematical analysis of vortices for \CGL was
first sketched in \cite{LX}, where it was presented as an
alternative approach (a regularized version) for the study of the
Gross-Pitaevskii equation. We believe however that the conclusion
regarding the dynamics of vortices for \CGL in \cite{LX} is
erroneous, and that Theorem \ref{thm:main} yields the corrected
version.

\medskip

 After the completion of this work we were informed that Spirn,
Kurzke, Melcher and Moser \cite{SpKu} independently obtained similar
results concerning the dynamics of vortices for \CGL in bounded,
simply connected domains.

\subsection{Renormalized energy and Cauchy Problem}

As mentioned in the introduction, for $d=\sum d_i \neq 0$ the
Ginzburg-Landau energy of $\uest(a_i,d_i)$ is infinite. It can
actually be computed that
\begin{equation*}
    \intr \frac{|\nabla |\uest(a_i,d_i) ||^2}{2}+ \frac{(1-|\uest(a_i,d_i)|^2)^2}{4\eps^2}\,dz <+\infty,
\end{equation*}
whereas as $|z|\to + \infty$,
\begin{equation*}
|\nabla \uest(a_i,d_i)|^2(z) \sim \frac{d^2}{|z|^2},
\end{equation*}
so that
\begin{equation*}
    \intr \frac{|\nabla \uest(a_i,d_i) |^2}{2} = +\infty.
\end{equation*}
The renormalized energy introduced in \cite{BS} is obtained by
substracting the diverging part of the gradient at infinity. More
precisely, given a smooth map $U_d$ such that
\begin{equation*}
U_d=\left(\frac{z}{|z|}\right)^d\qquad \textrm{on }\: \R\setminus
B(0,1),
\end{equation*}
we have as $|z|\to +\infty$
\begin{equation*}
|\nabla \uest(a_i,d_i)|^2\sim |\nabla U_d|^2
\end{equation*}
and one may define
\begin{equation}
\mathcal{E}_{\eps,U_d}(\uest(a_i,d_i)):=\lim_{R\to +\infty}
\int_{B(R)} e_\eps (\uest(a_i,d_i))-\frac{|\nabla
U_d|^2}{2}<+\infty.
\end{equation}
This definition extends to a larger class of functions, and is a
useful ingredient in solving the Cauchy problem.  Following
\cite{BS}, we define
\begin{eqnarray*}
\mathcal{V}=\{U\in L^{\infty} (\R,\mathbb{C}),\: \nabla^k U\in
L^2,\: \forall k\geq 2,\: (1-|U|^2)\in L^2,\: \nabla |U| \in L^2\}.
\end{eqnarray*}
In particular, the space $\mathcal{V}$ contains all the maps $\uest$
as well as the reference maps $U_d$. We state below and prove in the
Appendix global well-posedness in the class $\mathcal{V}+
H^1(\R)$\footnote{In \cite{GV}, the Cauchy problem in local spaces
is investigated for a more general class of complex Ginzburg-Landau
equations.}.
\begin{thm}
\label{cauchyproblem}
 Let $u_0=U+w_0$ be in $\mathcal{V}+H^1(\R)$.
Then there exists a unique global solution $u(t)$ to \CGL such that
$u(t)\in \{U\}+H^1(\R)$. If we write $u(t)=U+w(t)$, then $w$ is the
unique solution in $C^0(\mathbb{R}_+,H^1(\R))$ to
\begin{eqnarray}
\left\{ \begin{array}{l} \dsp (k_\eps +i)\dt w=\Delta w+f_{U}(w)
\\
w(0)=w_0, \end{array} \right. \label{pert}
\end{eqnarray}
where
\begin{equation*}
f_{U}(w)=\Delta U+\frac{1}{\eps^2}(U+w)(1-|U+w|^2).
\end{equation*}
In addition, $w$ satisfies
\begin{eqnarray*}
w\in L_{\loc}^1(\mathbb{R}_+,H^2(\R))\cap
L_{\loc}^{\infty}(\mathbb{R}_+^*,L^{\infty}(\R)),\: \dt w \in
L_{\loc}^1(\mathbb{R}_+,L^2(\R))
\end{eqnarray*}
and
\begin{equation*}
w\in C^\infty(\mathbb{R}_+^*,C^{\infty}(\R)).
\end{equation*}
Finally, the functional $E_{\eps,U}(u):=E_{\eps,U}(w)$ defined by
\begin{eqnarray*}
E_{\eps,U}(u)=\intr \frac{|\nabla w|^2}{2}-\intr \Delta U \cdot
w+\intr \frac{ (1-|U+w|^2)^2}{4\eps^2}
\end{eqnarray*}
satisfies
\begin{eqnarray*}
 \frac{d}{dt}E_{\eps,U}(u) =-k_\eps \intr |\dt w|^2
\,dx,\qquad \forall t\geq 0.
\end{eqnarray*}
\end{thm}
\medskip
 As a matter of fact, it follows from an integration by part
that if $u\in \{U\} + H^1(\R)$ is as in Theorem \ref{cauchyproblem}
and if $U$ satisfies in addition $|\nabla U(x)| \leq
\frac{C}{\sqrt{|x|}}$, then
\begin{equation*}
E_{\eps,U}(u(t))\equiv\mathcal{E}_{\eps,U}(u(t))=\lim_{R\to +\infty}
\int_{B(R)} (e_\eps (u(t))-\frac{|\nabla U|^2}{2})\,dx.
\end{equation*}

\medskip

The functions $\uest(a_i,d_i)$ are not $H^1$ perturbations one of
the other, even for fixed $d=\sum d_i$, unless some algebraic
relations involving the $a_i$'s and $d_i$'s hold. In order to handle
a class of functions containing them all, it is useful to introduce
the following equivalence relation on the set $\mathcal{V}$ :
\begin{eqnarray*}
&&\forall U,U'\in \mathcal{V},\qquad U\sim U' \: \: \textrm{iff} \\
&&\textrm{deg}_{\infty}(U)=\textrm{deg}_{\infty}(U')\: \:
\textrm{and} \: \: |\nabla U|^2-|\nabla U'|^2 \in L^1(\R).
\end{eqnarray*}
Denoting by $[U]$ the corresponding equivalence class of $U$, we
observe that for any configuration $(a_i,d_i)$ such that $\sum
d_i=d$, we have $\uest(a_i,d_i)\in [U_d]$. Therefore the space
$[U_d]+H^1(\R)$ contains in particular all $H^1$ perturbations of
all reference maps $\uest$ of degree $d$ at infinity.

\medskip

For a map $u$ in $[U_d]+H^1(\R)$, we may now define
\begin{eqnarray*}
\mathcal{E}_{\eps,[U_d]}(u):=\lim_{R\to +\infty}
\int_{B(R)}e_\eps(u)-\frac{|\nabla U_d|^2}{2},
\end{eqnarray*}
which is a finite quantity. Moreover, for any solution $u=u(t) \in
C^0([U_d]+ H^1(\R))$, we infer from Theorem \ref{cauchyproblem} that
\begin{eqnarray*}
\frac{d}{dt}\mathcal{E}_{\eps,[U_d]}(u)=\frac{d}{dt}\mathcal{E}_{\eps,U}(u)=-k_\eps
\intr |\dt u|^2.
\end{eqnarray*}
The dissipation of $\er(u(t))$ is therefore exactly the same as the
dissipation for the usual Ginzburg-Landau energy in the case of
bounded, simply connected domains.

\subsection{Statement of the result}\label{sect:resu}

In the sequel, $A_n$ denotes the annulus $B(2^{n+1})\setminus
B(2^n)$ for $n\in \mathbb{N}$, so that $\R=B(2^{n_0})\cup
(\cup_{n\geq n_0} A_n)$.
\begin{defi}
Let $a_1,\ldots, a_l$ be $l$ distinct points in $\R$, $d_i \in
\{-1,+1\}$ for $i=1,\ldots,l$ and set $d=\sum d_i$. Let
$(u_\eps)_{0<\eps<1}$ be a family of maps in $[U_d]+H^1(\R)$. We say
that $(u_\eps)_{0<\eps<1}$ is well-prepared with respect to the
configuration $(a_i,d_i)$ if there exist $R=2^{n_0}>\max |a_i|$ and
a constant $K_0>0$ such that\footnote{Here, $E_\eps(u,B) \equiv
\int_B e_\eps(u).$}
\begin{equation}
\lim_{\eps \to 0} \| Ju_\eps - \pi \ssum d_i \delta_{a_i}
\|_{W_0^{1,\infty}(B(R))^\ast}=0, \label{W1}
\tag{\scriptsize{WP$_1$}}
\end{equation}
\begin{equation}
 \sup_{0<\eps<1} E_{\eps}(u_{\eps},A_n)\leq K_0 \qquad \forall
n\geq n_0, \label{W2} \tag{\scriptsize{WP$_2$}}
\end{equation}
and
\begin{equation}
 \lim_{\eps \to 0} \left( \er(u_\eps)-\er(\uest(a_i,d_i))\right)=0.
\label{W3} \tag{\scriptsize{WP$_3$}}
\end{equation}
\end{defi}
 We can now state our main theorem as follows
\begin{thm}\label{thm:main}
\label{theorem1} Let $(\ueo)_{0<\eps<1}$ in $[U_d]+H^1(\R)$ be a
family of well-prepared initial data with respect to the
configuration $(a_i^0,d_i)$ with $d_i=\pm 1$, and let
$(\ue(t))_{0<\eps<1}$ in $C(\mathbb{R}^{+},[U_d]+H^1(\R))$ be the
corresponding solution of \CGL. Let $\{a_i(t)\}_{\{i=1,\ldots,l\}}$
denote the solution of the ordinary differential equation
\begin{equation}
\begin{cases}
\displaystyle \pi \dot{a}_i(t)=C_i \big(\delta d_i \mathbb{I}_2-
\mathbb{J}_2\big)\nabla
_{a_i} W,\qquad C_i=\frac{-d_i}{1+\delta^2}\\
a_i(0)=a_i,\qquad i=1,\ldots,l
\end{cases}
\label{systemepointsvortex}
\end{equation}
where
\begin{eqnarray*}
\mathbb{I}_2=\begin{pmatrix} 1&0\\0&1
\end{pmatrix},\qquad \mathbb{J}_2=\begin{pmatrix} 0&-1\\1&0
\end{pmatrix}
\end{eqnarray*}
and $W$ is the Kirchhoff-Onsager functional defined by
\begin{eqnarray*}
W(a_i,d_i)=-\pi \sum_{i\neq j} d_i d_j \log |a_i-a_j|.
\end{eqnarray*}
We denote by $[0,T^\ast)$ its maximal interval of existence. Then,
for every $t\in [0,T^\ast)$, the family $(\ue(t))_{0<\eps<1}$ is
well-prepared with respect to the configuration $(a_i(t),d_i)$.
\end{thm}


\numberwithin{equation}{section}

\section{Evolution formula for $\ue$}

In this section, we recall or derive a number of evolution formulae
involving quantities related to $u_\eps$ which we introduce now.

\subsection{Notations} Throughout this article, we identify $\mathbb{R}^2$ and $\mathbb{C}$.
Given $x=(x_1,x_2)\in \mathbb{R}^2$, we set $x^{\bot}=(-x_2,x_1)$,
which in complex notations reads $x^{\bot}=ix$. For $z$ and $z'\in
\mathbb{C}$, $z\cdot z'=\mathrm{Re}(z\overline{z'})$ denotes the
scalar product and $z\times z'=z^{\bot}\cdot z'=-\mathrm{Im} (z
\overline{z'})$ the exterior product of $z$ and $z'$ in
$\mathbb{R}^2$. For a map $u:\R \rightarrow \mathbb{C}$, we denote
by
\begin{eqnarray*}
j(u)=u\times \nabla u=iu\cdot \nabla u=u^{\bot} \cdot \nabla u
\end{eqnarray*}
the linear momentum and
\begin{eqnarray*}
J(u)=\partial_1 u\times \partial_2u =\textrm{det}(\nabla u)
\end{eqnarray*}
the Jacobian of $u$. For $u\in H_{\loc}^1(\R)$, it can be checked
that $J(u)=\frac{1}{2}\textrm{curl} j(u) $ in the distribution
sense. On the set where $u$ does not vanish, we have for $k=1,2$
\begin{eqnarray*}
\partial_k u = \partial_k u\cdot \frac{u}{|u|}
\frac{u}{|u|}+\partial_k u \cdot \frac{iu}{|u|} \frac{iu}{|u|}.
\end{eqnarray*}
This yields
\begin{eqnarray}
\label{modules2} \partial_k u = \partial_k |u| \frac{u}{|u|} +
\frac{j_k(u)}{|u|} \frac{u^{\bot}}{|u|},
\end{eqnarray}
 hence we have
\begin{eqnarray}
\label{modules3}
\partial_k u\cdot \partial_l u =\partial_k |u|\partial_l
|u|+\frac{j_k(u)j_l(u)}{|u|^2}
\end{eqnarray}
and it follows that
\begin{eqnarray}
\label{modules} |\nabla u|^2 =|\nabla |u| |^2
+\frac{|j(u)|^2}{|u|^2}.
\end{eqnarray}
The Hopf differential of $u$ is defined as
\begin{eqnarray*}
\omega(u)=|\partial_1 u|^2-|\partial_2 u|^2-2i \dun u\cdot \ddeux
u=4
\partial_z u \overline{\partial_{\overline{z}}u}.
\end{eqnarray*}
It follows from \eqref{modules3} that $\omega(u)$ may be rewritten
in terms of the components of $\nabla |u|$ and $j(u)$ as
\begin{equation}
\begin{split}
\label{omega} \omega(u)=\partial_1|u|^2&-\partial_2 |u|^2-2i\,
\partial_1 |u| \partial_2 |u|\\&+\frac{1}{|u|^2}\big(j_1^2(u)-j_2^2(u)-2i \,j_1(u) j_2(u)\big).
\end{split}
\end{equation}
We recall that the Ginzburg-Landau energy density is defined by
\begin{eqnarray*}
e_\eps(u)=\frac{|\nabla
u|^2}{2}+\frac{(1-|u|^2)^2}{4\eps^2}=\frac{|\nabla u|^2}{2}+V(u),
\end{eqnarray*}
and we set
\begin{eqnarray*}
\mu_\eps(u)=\frac{e_\eps(u)}{|\log \eps|}.
\end{eqnarray*}
 In view of \eqref{modules}, we then have
\begin{equation}
\label{energiesmodules} e_\eps(u)=e_\eps(|u|)+
\frac{|j(u)|^2}{|u|^2}.
\end{equation}
Finally, we write the right-hand side in \CGL  as
\begin{eqnarray*}
\nabla E(u)=\nabla E_\eps(u)=\Delta u +\frac{1}{\eps^2} u(1-|u|^2).
\end{eqnarray*}


\subsection{Evolution formulae involving the Jacobian and the energy
density}

For a smooth map $u$ in space-time, direct computations by
integration by part yield for the energy
\begin{equation}
\begin{split}
\label{evolutionenergie} \frac{d}{dt} \intr
e_\eps(u)\varphi\,dx=-\intr \dt u \cdot &\nabla E(u)\varphi
\,dx\\
&-\intr \nabla \varphi \cdot (\dt u\cdot \nabla u)\,dx
\end{split}
\end{equation}
and for the Jacobian
\begin{equation}
\label{evolutionjacobien} \frac{d}{dt} \intr J(u) \chi\,dx=-\intr
\nabla^{\bot} \chi \cdot (\dt u^{\bot}\cdot \nabla u)\,dx,
\end{equation}
where $\chi,\varphi\in \mathcal{D}(\R)$.

Also, for any vector field $\vec{X}\in C^1(\R,\mathbb{C})$ we have
(see e.g. \cite{BOS1})
\begin{eqnarray*}
\intr \vec{X}\cdot(\nabla E(u) \cdot \nabla u)\,
dx=2\intr\mathrm{Re}\Big( \omega(u)\frac{\partial
\vec{X}}{\partial{\overline{z}}}\Big)\,dz-\intr V(u)\nabla\cdot
\vec{X}\,dx.
\end{eqnarray*}
In particular, the choice of $\vec{X}=\nabla \varphi$ or
$\vec{X}=\nabla^{\bot}\chi=i\nabla \chi$ leads to
\begin{eqnarray*}
\intr \nabla \varphi \cdot(\nabla E(u) \cdot \nabla u)\,
dx=2\intr\mathrm{Re}\Big(\omega(u)\frac{\partial^2
\varphi}{\partial{\overline{z}}^2}\Big)\,dz-\intr V(u)\Delta
\varphi\,dx
\end{eqnarray*}
and
\begin{equation}
\label{pohozaev} \intr \nabla^{\bot} \chi \cdot(\nabla E(u) \cdot
\nabla u)\, dx=-2\intr\mathrm{Im}\Big( \omega(u)\frac{\partial^2
\chi}{\partial{\overline{z}}^2}\Big)\,dz.
\end{equation}

We next consider a solution $u$ of \CGL, which is smooth in view of
Theorem \ref{cauchyproblem}. In this case, $\nabla E(u)$ and $\dt u$
are related by
\begin{eqnarray}
\label{eve} \dt u=\frac{1}{\alpha_\eps} \nabla E(u)=\beta_\eps
\nabla E(u),
\end{eqnarray}
where $\dsp \alpha_\eps=\keps+i=k_\eps+i$. Using \eqref{eve} in
\eqref{evolutionenergie} and \eqref{evolutionjacobien}, we obtain
\begin{eqnarray*}
\frac{d}{dt} \intr e_\eps(u)\varphi\,dx=-\frac{\delta}{|\log \eps|}
\intr |\dt u|^2 \varphi \,dx -\intr \nabla \varphi \cdot (\beta_\eps
\nabla E(u)\cdot \nabla u)\,dx
\end{eqnarray*}
and
\begin{eqnarray*}
\frac{d}{dt} \intr J(u) \chi\,dx=-\intr \nabla^{\bot} \chi \cdot (i
\beta_\eps \nabla E(u)\cdot \nabla u)\,dx.
\end{eqnarray*}
In order to get rid of the terms of the form $\dsp \intr
\vec{X}\cdot (i\nabla E(u)\cdot \nabla u)$, we compute
\begin{eqnarray*}
\frac{d}{dt} \intr (b J(u) \chi-a e_\eps(u)\varphi)
\end{eqnarray*}
where $\beta_\eps=a+ib$. This yields
\begin{equation}
\begin{split}
\label{evolenergieetjacobien} \frac{d}{dt} \intr b J(u) \chi- a
e_\eps(u)&\varphi=(b^2+a^2) \intr\nabla^{\bot} \chi \cdot (\nabla
E\cdot \nabla
u)+a k_\eps \intr |\dt u|^2 \,dx \nonumber\\
&+\intr(\nabla \varphi-\nabla^{\bot} \chi) \cdot (a(a+ib)\nabla
E\cdot \nabla u).
\end{split}
\end{equation}
Since $\dsp a=\frac{k_\eps}{k_\eps^2+1}$ and $\dsp
b=\frac{-1}{k_\eps^2+1}$, we can multiply
\eqref{evolenergieetjacobien} by $k_\eps^2+1$. Using finally
\eqref{pohozaev}, we obtain
\begin{prop}
\label{evolution} Let $u$ solve \CGL. Then for all $\varphi, \chi\in
\mathcal{D}(\R)$,
\begin{equation}
\begin{split}
\label{formuleevolution} \frac{d}{dt} \intr J(u) \chi+ k_\eps
e_\eps(u)\varphi&=-k_\eps^2 \intr |\dt u|^2 \varphi+2\intr
\mathrm{Im} \Big(\omega(u) \frac{\partial^2 \chi}{\partial
\overline{z}^2}\Big) \nonumber\\
&+R_\eps(t,\varphi,\chi,u),
\end{split}
\end{equation}
where the remainder $R_\eps$ is defined by
\begin{eqnarray*}
R_\eps(t,\varphi,\chi,u)=-k_\eps \intr (\nabla \varphi-\nabla^{\bot}
\chi)\cdot(\beta_\eps\nabla E(u)\cdot \nabla u)
\end{eqnarray*}
or equivalently
\begin{eqnarray*}
 R_\eps(t,\varphi,\chi,u) =-k_\eps \intr (\nabla
\varphi-\nabla^{\bot} \chi)\cdot(\dt u\cdot \nabla u).
\end{eqnarray*}
\end{prop}

Proposition \ref{evolution} allows to derive formally the motion law
for the vortices. Indeed, assume that we have
\begin{eqnarray*}
Ju_\eps (t)\to \pi \ssum d_i \delta_{\ai(t)},
\end{eqnarray*}
\begin{eqnarray*}
\mu_\eps(u_\eps)(t)\to \pi \ssum \delta_{a_i(t)}
\end{eqnarray*}
and $u_\eps(t)$ is close in some sense to $\dsp \uest(a_i(t),d_i)$,
and therefore to $\dsp u^*(a_i(t),d_i)$, where
\begin{eqnarray*}
u^*(a_i,d_i)=\prod_{i=1}^l \left(\frac{z-a_i}{|z-a_i|}\right)^{d_i}.
\end{eqnarray*}
We use Proposition \ref{evolution} with $u$ formally replaced by
$\dsp u^*(a_i(t),d_i)$ and with choices of test functions $\varphi$
and $\chi$ which are localized and affine near each point $a_i(t)$
and satisfy $\nabla \varphi=\nabla^{\bot} \chi$ there, so that both
terms $k_\eps^2 \intr |\dt u|^2 \varphi$ and
$R_\eps(t,\varphi,\chi,u_\eps)$ vanish in the limit $\eps \to 0$.
Using the formula (see \cite{BJS})
\begin{equation*}
2 \intr\mathrm{Im} \left(\omega\big(u^*(a_i(t),d_i)\big)
\frac{\partial^2 \chi}{\partial \overline{z}^2}\right)=-\pi
\sum_{j\neq i} d_i d_j \frac{(a_i-a_j)^{\bot}}{|a_i-a_j|^2}\cdot
\nabla \chi (a_i),
\end{equation*}
we then obtain that for each $i$
\begin{eqnarray*}
\pi d_i \dot{a_i}(t)\cdot \nabla \chi(a_i)+\delta \pi \dot{a_i}(t)
\cdot \nabla \varphi(a_i)=-\pi \sum_{j\neq i} d_i d_j
\frac{(a_i-a_j)^{\bot}}{|a_i-a_j|^2}\cdot \nabla \chi (a_i).
\end{eqnarray*}
Taking into account the fact that $\nabla
\varphi(a_i)=\nabla^{\bot}\chi(a_i)$, we infer that
\begin{eqnarray*}
\pi \big(d_i\dot{a_i}(t)-\delta \dot{a_i}^{\bot}(t)\big)\cdot \nabla
\chi(a_i)=-\pi \sum_{j\neq i} d_i d_j
\frac{(a_i-a_j)^{\bot}}{|a_i-a_j|^2}\cdot \nabla \chi (a_i),
\end{eqnarray*}
which yields the ODE \eqref{systemepointsvortex}.

\medskip

In Section 4 and 5, in order to give a rigorous meaning to the
previous computations, we will prove the convergence of the
Jacobians and of the energy densities to the weighted sum of dirac
masses mentioned above, and then show that both the energy
dissipation $k_\eps^2 \intr |\dt u|^2 \varphi $ and the remainder
$R_\eps(t,\varphi,\chi,u_\eps) $
 vanish asymptotically when $\eps$ tends to zero. Finally, we will establish a control
 of $\omega(u^*(a_i),d_i)-\omega(u_\eps)$ or equivalently of $
 \omega(\uest(a_i),d_i)-\omega(u_\eps)$ in
 $L_{\loc}^1(\R/\{a_i(t)\})$.


\section{Some results on the renormalized energy}

In this section, we study the link between the energy $\er$ and the
usual Ginzburg-Landau energy on large balls. This may be achieved
for maps having uniform bounded energy on large annuli by defining a
degree at infinity.

\subsection{Energy at infinity and topological degree at infinity}

Let $A$ be the annulus $B(2)/B(1)$. We define
\begin{eqnarray*}
T_d=\{u\in H^1(A)\: \textrm{s.t.} \: \exists B\subset B(u),\:|B|\geq
\frac{3}{4},\: \forall r\in B,\: \textrm{deg}(u,\partial B(r))=d\}
\end{eqnarray*}
and
\begin{eqnarray*}
E_\eps^\Lambda=\{u\in H^1(A) \: \textrm{s.t.} \:
E_\eps(u,A)<\Lambda\}.
\end{eqnarray*}
The topological sector of degree $d$ is then defined as
\begin{eqnarray*}
S_{d,\eps}^\Lambda=E_\eps^\Lambda\cap T_d.
\end{eqnarray*}
The following Theorem was proved in \cite{A}.

\begin{thm}
\label{Almeida} For all $\Lambda>0$, there exists $\eps_\Lambda>0$
such that for every $0<\eps<\eps_\Lambda$, we have
\begin{equation*}
E_\eps^\Lambda=\bigcup_{d\in \mathbb{Z}}S_{d,\eps}^\Lambda.
\end{equation*}
\end{thm}
In the sequel of this section, we fix $\Lambda>\Lambda_d= \pi d^2
\log(2)$ and we set
\begin{equation*}
S_d(\Lambda)\equiv S_{d,\eps_\Lambda}^\Lambda,
\end{equation*}
so that in particular the map $U_d$ belongs to $S_d(\Lambda)$.

Let $u\in [U_d]+H^1(\R)$ and for $k\in \mathbb{N}$, set $u_k : z\in
A \mapsto u(2^k z)$. By scaling, we find that for every
$0<\eps<\eps_{\Lambda}$, the map $u_k$ belongs to
$E_{\eps_\Lambda}^{\Lambda}$ for $k\geq k(\eps)$ sufficiently large
and therefore to some topological sector $S_{d(k),\eps}^\Lambda$.
Thanks to the uniform bound for the energy $E_{\eps}(u_k,A)$ for
large $k$, this degree is necessarily identically equal to $d$.

\begin{prop}[\cite{BJS}, Corollary 3.1]
Let $d\in \mathbb{Z}$ and $\Lambda>\Lambda_d$. For any $u\in
[U_d]+H^1(\R)$, there exists an integer $n\in \mathbb{N}^*$ such
that for all $k\geq n$, the map $u_k : z\in A \mapsto u(2^k z)$
belongs to the topological sector $S_d$. We denote by $n(u,\Lambda)$
the smallest integer having this property. The map $u\mapsto
n(u)=n(u,\Lambda)$ is continuous.
\end{prop}

We first have the following

\begin{lemme}
\label{energyatinfinityanddegree} Let $\Lambda>\Lambda_d$ be given.
Let $u\in [U_d]+H^1(\R)$ and assume that there exists $n_0\in
\mathbb{N}^*$  such that for all $n\geq n_0$,
\begin{eqnarray*}
E_{\eps_\Lambda}(u,A_n) < \Lambda.
\end{eqnarray*}
Then we have $n(u,\Lambda)\leq n_0$.
\end{lemme}

The definition of $n(u)$ allows to obtain a lower bound for $\er$ on
annuli.

\begin{lemme} [\cite{BJS}, Lemma 3.1]
\label{almostminimizing} Let $d\in \mathbb{Z}$ and $u\in
[U_d]+H^1(\R)$. Then, for any $k\geq n(u)$, we have for
$\eps<\eps_{\Lambda}$
\begin{equation*}
\int_{A_k} [e_\eps(u)-\frac{|\nabla U_d|^2}{2} ] \geq -C2^{-2k}
\eps^2.
\end{equation*}
\end{lemme}

Lemma \ref{linkenergies} below provides an upper bound for the
Ginzburg-Landau energy on sufficiently large balls in terms of the
excess energy $\er(u)-\er(\uest)$. This will enable us to rely on
results holding for the Ginzburg-Landau functional in bounded
domains in the proof of Theorem \ref{theorem1}.

\begin{lemme}[\cite{BJS}, Lemma 3.2]

\label{linkenergies} Let $d\in \mathbb{Z}$, $u\in [U_d]+H^1(\R)$,
$a_1,\ldots,a_l\in \R$ and $d_1,\ldots,d_l\in \mathbb{Z}^*$ such
that $d=\sum d_i$. Then, for $k\geq 1+ \max\{
\log_1|a_1|,\ldots,\log_2|a_l|,n(u)\}$ and $R=2^k$ we have
\begin{eqnarray*}
\int_{B(R)} e_\eps(u)-e_\eps(\uest(a_i,d_i))\leq
\er(u)-\er(\uest(a_i,d_i))+\frac{C}{R},
\end{eqnarray*}
where $C$ depends only on $l$ and $d$.
\end{lemme}

\subsection{Explicit identities for the reference map $\uest$}

We present here an account of some classical identities for the
energy of $\uest$, which are borrowed from \cite{BJS}.

In the sequel, we consider a configuration $(a_i,d_i)$ with $d_i\in
\mathbb{Z}^*$ and we set $d=\sum d_i$. We begin with an explicit
expansion near each vortex $a_j$.

\begin{lemme}
\label{energynearthecore} For $j\in\{1,\ldots,l\}$ and $0<\eps<1$,
\begin{eqnarray*}
\int_{B(a_j,r)} e_\eps(\uest(a_i,d_i)) =\pi d_j^2
\log(\frac{r}{\eps})+\gamma(|d_j|)+O(\frac{r}{r_a})^2+O(\frac{\eps}{r})^2
\end{eqnarray*}
where $\gamma(|d_j|)$ is some universal constant. 
\end{lemme}
On the other hand, $\uest(a_i,d_i)$ behaves as $u^*(a_i,d_i)$ away
from the vortices, so its energy on $\Omega_{R,r}=B(R)\setminus \cup
B(a_j,r)$ is close to the energy of $u^*(a_i,d_i)$ on $\Omega_{R,r}$
which we can compute explicitely (see \cite{BBH}). Combining the
previous expansions, we obtain

\begin{prop}
\label{energyonball}
 Let
\begin{eqnarray*}
r_a=\frac{1}{8} \min_{i\neq j} \{|a_i-a_j|\},\qquad R_a=\max
\{|a_i|\}.
\end{eqnarray*}
Then for $R>R_a+1$,we have as $\eps\to 0$
\begin{eqnarray*}
\int_{B(R)} e_\eps(\uest(a_i,d_i))=\pi \sum _{i=1}^l d_i^2 \leps
+W(a_i,d_i) +\sum_{i=1}^l \gamma(|d_i|)\\\hspace*{-2em}+ \pi d^2
\log R + O(\frac{R_a}{R})+o_\eps(1).
\end{eqnarray*}
\end{prop}
We observe that as $R\to + \infty$, we have $\pi \log^2 R \sim \int
_{B(R)}\frac{|\nabla U_d|^2}{2}$. This yields the following
expansion for the renormalized energy
\begin{coro}
\label{renormalizedenergy} When $\eps \to 0$, the following holds
\begin{equation*}
\begin{split}
\er(\uest(a_i,d_i))=\pi \sum _{i=1}^l d_i^2 \leps &+W(a_i,d_i)\\
&+\sum_{i=1}^l \gamma(|d_i|)-\int_{B(1)}\frac{|\nabla
Ud|^2}{2}+o_\eps(1).
\end{split}
\end{equation*}
\end{coro}
Concerning the energy on annuli, we finally quote the following
result
\begin{lemme}
\label{energyonannulii}
 For $R>R_a$, we have
\begin{equation*}
\int_{B(2R)/B(R)} e_\eps(\uest(a_i,d_i))=\pi d^2 \log 2
+O\big(\frac{R_a}{R}\big)
\end{equation*}
or, in view of the properties of $U_d$ at infinity,
\begin{equation*}
\int_{B(2R)/B(R)} e_\eps(\uest(a_i,d_i))=\int_{B(2R)/B(R)}
\frac{|\nabla U_d|^2}{2} +O\big(\frac{R_a}{R}\big).
\end{equation*}
\end{lemme}

\section{Coercivity}

In this section, we supplement some results from \cite{BJS} and
\cite{JS1} with estimates which we will later need. These results
establish precise estimates in various norms for maps $u$
 being close to $\uest(a_i,d_i)$ in
terms of the excess energy with respect to the configuration
$(a_i,d_i)$. For a map $u\in [U_d]+H^1(\R)$ and a given
configuration $(a_i,d_i)$ with $d_i=\pm 1$, we define this excess
energy $\Sigma_\eps$ as
\begin{equation*}
\Sigma_\eps=\Sigma_\eps(a_i,d_i)=\er(u)-\er(\uest(a_i,d_i)).
\end{equation*}
We also set
\begin{equation*}
r_a=\frac{1}{8}\min_{i\neq j} \{|a_i-a_j|\},\qquad
R_a=\max_{i=1,\ldots,l} \{|a_i|\}.
\end{equation*}

\begin{thm}
\label{coercivitytheorem} Let $r\leq r_a$ and $2^{n_0}=R_0>R_a$ such
that $\cup_{i=1}^l B(a_i,r)\subset B(R_0)$. Then there exist
$\eps_0$ and $\eta_0$ depending only on $l$, $r$, $r_a$, $R_a$,
$R_0$ satisfying the following property. For all $u\in
[U_d]+H^1(\R)$ such that
\begin{equation}
\label{c1} \eta=\|Ju-\pi \ssum d_i
\delta_{a_i}\|_{W_0^{1,\infty}(B(R))^*} \leq \eta_0
\end{equation}
and
\begin{equation}
\label{c2}
 2^{n(u)}\leq R_0,
\end{equation}
then for $\eps \leq \eps_0$ we have
\begin{equation}
\begin{split}
\label{coercivity} \int_{B(R_0)\setminus \cup B(a_i,r)}
e_\eps(|u|)+\frac{1}{8} \Big|\frac{j(u)}{|u|}
-j(u^*(a_i,d_i))\Big|^2 \leq \Sigma_\eps
+C(\eta,\eps,\frac{1}{R_0}),
\end{split}
\end{equation}
where $C$ is a continuous function on $\mathbb{R}^3$ vanishing at
the origin. Furthermore, there exist points $b_i\in B(a_i, r/2)$
such that
\begin{equation}
\label{concentrationjacobien} \|Ju-\pi \ssum d_i
\delta_{b_i}\|_{W_0^{1,\infty}(B(R_0))^*}\leq f(R_0,\Sigma_\eps)
\eps |\log \eps|
\end{equation}
and
\begin{equation}
\label{concentrationenergie}
 \|\mu_\eps(u)-\pi \ssum
\delta_{b_i}\|_{W_0^{1,\infty}(B(R_0))^*} \leq
\frac{g(R_0,r,r_a,\Sigma_\eps)}{|\log \eps|},
\end{equation}
where $f$ and $g$ are continuous functions on $\R$ and
$\mathbb{R}^4$.
\end{thm}

\begin{proof}
Except for the energy concentration \eqref{concentrationenergie},
each of the other statements are already proved in Theorem 6.1 of
\cite{BJS}. We first infer from \eqref{c1} that for all $i$
$\|Ju-\pi d_i \delta_{a_i}\|_{W_0^{1,\infty}(B(a_i,r))^*}\leq
\eta_0$. If $\eta_0$ is sufficiently small with respect to $r$ this
gives in view of Theorem 3 in \cite{JS1} $ K_0^i \geq C(r),$ where
$K_0^i$ is the local excess energy near the vortex $i$ defined by $
K_0^i=\int_{B(a_i,r)} e_\eps(u)-\pi \log (\frac{r}{\eps})$. It
follows that
\begin{equation*}
\int_{B(a_i,r)} e_\eps(u)\leq \int_{B(R_0)} e_\eps(u) -\pi(l-1)\leps
-C(r).
\end{equation*}
On the other hand, since $n(u)\leq n_0$, we have according to Lemma
\ref{linkenergies} and Proposition \ref{energyonball}
\begin{equation*}
\int_{B(R_0)} e_\eps(u)\leq\int_{B(R_0)} e_\eps(\uest(a_i,d_i))
+\Sigma_\eps +\frac{C}{R_0} \leq \pi l\leps +\Sigma_\eps +C.
\end{equation*}
This first implies that $K_0^i \leq C+ \Sigma_\eps$. Also, replacing
$r$ by $3r/4$ we see that $ \int _{B(R_0)\setminus \cup B(a_i,3r/4)}
\mu_\eps(u)\leq (C+\Sigma_\eps)|\log \eps|^{-1}$, where $C$ only
depends on $R_0,r,r_a,R_a$.

Now, according to Theorem 2' in \cite{JS1}, the energy density
$\mu_\eps(u)$ on $B(a_i,r)$ concentrates at the point $b_i\in
B(a_i,r/2)$ where $J(u_\eps)$ concentrates. From Theorem 3.2.1 in
\cite{CJ2} and the estimate for $K_0^i$ it follows that
\begin{equation*}
\|\mu_\eps(u)-\pi \delta_{b_i}\|_{W_0^{1,\infty}(B(a_i,r))^*} \leq
\frac{f(\Sigma_\eps,C)}{|\log \eps|}.
\end{equation*}

Combining the above and the upper bound for the energy density
outside the vortex balls finally yields
\eqref{concentrationenergie}.
\end{proof}


\section{Convergence to Lipschitz vortex paths}

In this section, we establish compactness for the Jacobians and the
energy densities under weaker assumptions on the initial excess
energy. Instead of assuming that this excess energy vanishes
initially, we only require that it is uniformly bounded with respect
to $\eps$.

\begin{thm}
\label{theoreme1bis} Let $(a_i^0,d_i)$ with $d_i=\pm 1$ be a
configuration of vortices. Let $R=2^{n_0}$ and $(\ueo)_{0<\eps<1}$
in $[U_d]+H^1(\R)$ such that
\begin{equation}
\lim_{\eps \to 0} \| Ju_\eps - \pi \ssum d_i \delta_{a_i^0}
\|_{W_0^{1,\infty}(B(R))^\ast}=0, \label{WWP1}
\tag{\scriptsize{WP$_1$}}
\end{equation}
\begin{equation}
\sup_{0<\eps<1} E_{\eps}(u_{\eps},A_n)\leq K_0, \qquad \forall n\geq
n_0 \label{WWP2}, \tag{\scriptsize{WP$_2$}}
\end{equation}
and
\begin{equation}
\sup_{0<\eps < 1} \Big( \er(u_\eps)-\er(\uest(a_i,d_i))\Big)\leq
K_1. \label{WWP3}\tag{\scriptsize{WP$_{3'}$}}
\end{equation}
Then there exist $R'=2^{n_1}$ and $T>0$ depending only on
$K_1,R,r_a$ and $R_a$, a sequence $\eps_k \to 0$ and $l$ Lipschitz
paths $b_i : [0,T]\rightarrow \R$ starting from $a_i^0$ such that
\begin{equation}
\label{concentrationjacobien2} \sup_{t\in [0,T]}\|Ju_{\eps_k}(t)-\pi
\sum_{i=1}^l d_i \delta_{b_i(t)}\|_{W_0^{1,\infty}(B(R'))^*}\to
0,\qquad k\to + \infty
\end{equation}
and
\begin{equation}
\label{concentrationenergie2} \sup_{t\in [0,T]}
\|\mu_{\eps_k}(u_{\eps_k})(t)-\pi \sum_{i=1}^l
\delta_{b_i(t)}\|_{W^{1,\infty}(B(R'))^*} \to 0,\qquad k\to +\infty.
\end{equation}
Moreover, there exist a constant $C_0>0$ depending only on $r_a, R,
K_1$ and $K_0$ and a constant $C_1>0$ depending on $r_a, R$ and $
K_1$ such that for all $t\in [0,T]$ and for $k\in \mathbb{N}$,
\begin{equation}
\label{energieanneau}
 E_{\eps_k}(u_{\eps_k}(t), A_n) \leq C_0,\qquad \forall n\geq n_1
\end{equation}
and
\begin{eqnarray}
\label{energieresteO}
 \mathcal{E}_{\eps_k,
[U_d]}\left(u_{\eps_k}(t)\right)-\mathcal{E}_{\eps_k,
[U_d]}\left(u_{\eps_k}^\ast(b_i(t),d_i)\right) \leq C_1.
\end{eqnarray}
\end{thm}

\begin{proof}
The proof is very similar to the proof of Theorem 4 in \cite{BJS}.
In the sequel, $C$ will stand for a constant depending only on $r_a,
R, R_a$ and $K_1$.

\medskip

We first consider $\Lambda>K_0$. Thanks to Lemma
\ref{energyatinfinityanddegree} and \eqref{WWP2}, there exists
$\eps_\Lambda>0$ such that for all $\eps<\eps_\Lambda$, we have
$n(u_\eps^0)=n(u_\eps^0,\Lambda)\leq n_0$. We fix such a $\Lambda$
and from now on only consider $\eps<\eps_\Lambda$.

\medskip

We next introduce $R'=\max(R,R_a+r_a)$ and define $n_1 \geq n_0$ as
the smallest integer for which $2^{n_1} \leq R'$. In the remainder
of the proof, we will assume without loss of generality that
$R'=2^{n_1}$ and we will write $\|\cdot \|$ instead of $\|\cdot \|
_{W_0^{1,\infty}(B(R'))^*}$.
 Our aim is to apply Theorem \ref{coercivitytheorem} to each $u_\eps(t)$ for the choice $r=r_a$ and $R_0=R'$.
   Let $\eta_0$ and $\eps_0$ be the constants provided
by Theorem \ref{coercivitytheorem} for this choice. First, thanks to
\eqref{WWP2} and \eqref{WWP3} it turns out that the convergence in
\eqref{WWP1} still holds on the larger ball $B(R')$ (see the proof
of Lemma 7.3 in \cite{BJS}). Therefore, since $t\mapsto
Ju_\eps(t)\in L^1(B(R'))$ is continuous for each $\eps$, there
exists a time $T_\eps>0$ such that
\begin{equation}
\label{jac} \|Ju_{\eps}(s)-\pi \ssum d_i
\delta_{a_i^0}\|<\eta_0,\qquad \forall s\in [0,T_\eps).
\end{equation}
We take $T_\eps$ to be the maximum time smaller than $T^\ast$ having
this property, where $T^\ast$ is defined in Theorem \ref{theorem1}.

On the other hand, since $t\mapsto E_\eps(u_\eps(t),A_n)$ is
continuous uniformly with respect to $n$ and $\Lambda>K_0$, we infer
from \eqref{WWP2} that there exists $T'_\eps>0$ such that for
$s\in[0,T'_\eps]$
\begin{equation*}
E_\eps(u_\eps(s),A_n)<\Lambda,\quad \forall n\geq n_1,
\end{equation*}
so according to Lemma \ref{energyatinfinityanddegree} we have
$n(u_\eps(s))\leq n_1$ for $s\in [0,T'_\eps]$.

We claim that there exists a constant $D$ depending on $K_1$, $r_a$,
$R$ and $K_0$ such that for all $s\in [0,\min(T_\eps,T'_\eps))$,
\begin{equation}
\label{eb} E_\eps(u_\eps(s),A_n) \leq D,\qquad \forall n\geq n_1.
\end{equation}

Consequently, if we assume from the beginning that
\begin{equation*}
\Lambda> \max(K_0,D),
\end{equation*}
then it follows from Lemma \ref{energyatinfinityanddegree} that
$n(u_\eps(s)) \leq n_1$ on $[0,\min(T_\eps,T'_\eps)]$.
  Therefore $T'_\eps > T_\eps$ and the topological degrees of the maps $u_\eps(t)$ at infinity remain uniformly
  bounded by $n_1$ as long as their Jacobians satisfy \eqref{jac}.

  \medskip

\noindent \emph{Proof of \eqref{eb}.} As in \cite{BJS}, we decompose
for each $n\geq n_1$
$E_\eps(u_\eps(t),A_n)-E_\eps(\uest(\aio,d_i),A_n)$ as
\begin{equation*}
\begin{split}
\sum_{\substack{k=n_1\\ k\neq n}}^{+ \infty}& \left(
E_\eps(\uest(\aio,d_i),A_k)-E_\eps(u_\eps(t),A_k)\right)\\
&+E_\eps\left(\uest(\aio,d_i),B(R')\right)-E_\eps\left(u_\eps(s),B(R')\right)\\
&+\er\left(u_\eps(s)\right)-\er\left(\uest(a_i^0,d_i)\right).
\end{split}
\end{equation*}

We first handle each term of the sum in the right-hand side. In view
of Lemmas \ref{almostminimizing} and \ref{energyonannulii}, we have
for $k\geq n_1$
\begin{equation*}
\begin{split}
E_\eps(u_\eps(t),A_k)&\geq -C\eps^2 2^{-2k} +\int_{A_k}
\frac{|\nabla U_d|^2}{2}\\
&\geq  E_\eps(\uest(\aio,d_i),A_k)-C(R_a)2^{-k}-C\eps^2 2^{-2k},
\end{split}
\end{equation*}
so we deduce that
\begin{equation*}
\sum_{\substack{k=n_1\\ k\neq n}}^{+\infty} \big(
E_\eps(\uest(\aio,d_i),A_k)-E_\eps(u_\eps(t),A_k)\big)\leq C.
\end{equation*}
Next, we infer from the definition of $T_\eps$ and Theorem 3 in
\cite{JS1}  that $\int_{B(a_i^0,r_a)} e_\eps(u_\eps(s)) \geq \pi
\leps -C$. Observe that $R'$ is chosen so that $\cup B(a_i^0,r_a)
\subset B(R')$, so this leads to
\begin{eqnarray*}
E_\eps\left(u_\eps(s),B(R')\right)\geq \pi l \leps -C.
\end{eqnarray*}
Using Proposition \ref{energyonball}, we thus find
\begin{eqnarray}
\label{upperbound3}
 E_\eps\left(\uest(\aio,d_i),B(R')\right)-E_\eps\left(u_\eps(s),B(R')\right)\leq C.
\end{eqnarray}
Finally, we define
$\Sigma_\eps^0(s):=\er\left(u_\eps(s)\right)-\er\left(\uest(a_i^0,d_i)\right)$.
Since $\er\left(u_\eps(t)\right)$ is non-increasing, we obtain in
view of \eqref{WWP3}
\begin{eqnarray*}
 \Sigma_\eps^0(s)\leq
\er(u_\eps^0)-\er\left(\uest(a_i^0,d_i)\right)\leq K_1,
\end{eqnarray*}
and \eqref{eb} follows.

\medskip


We may now apply Theorem \ref{coercivitytheorem} to each $u_\eps(t)$
on $[0,T_\eps]$. This provides points $\dsp b_i^{\eps}(s)\in
B(a_i^0,\frac{r_a}{2})$ for $0\leq s\leq T_\eps$. Since
$\Sigma_\eps^0(s)\leq K_1$, estimate \eqref{coercivity} turns into
\begin{eqnarray*}
\int_{\Omega_{R',r_a}} e_\eps(|u_\eps(s)|)+\frac{1}{8}
\left|\frac{j(u_\eps(s))}{|u_\eps(s)|} -j(u^*(a_i^0,d_i))\right|^2
\leq C,
\end{eqnarray*}
where $\Omega_{R',r_a}=B(R')\setminus\cup B(a_i^0,r_a)$. Also, we
have by \eqref{omega} and \eqref{energiesmodules}
\begin{eqnarray}
\label{upperbound2} \int_{\Omega_{R',r_a}} e_\eps(u_\eps(s)) \leq C
\end{eqnarray}
and
\begin{eqnarray}
\label{upperboundomega} \|
\omega(u_\eps(s))\|_{L^1(\Omega_{R',r_a})} \leq C,
\end{eqnarray}
where $C=C(R,r_a,K_1)$. For notation convenience, we may now write
$\mu_\eps$ instead of $\mu_\eps(u_\eps)$.

\medskip

In the sequel, given any configuration $(a_i,d_i)$, we denote by
$\mathcal{H}(a_i)$ the set of functions $\chi, \varphi \in
\mathcal{D}(\R)$ such that
\begin{eqnarray*}
\chi=\sum_{i=1}^l \chi_i,\qquad \varphi=\sum_{i=1}^l \varphi_i,
\end{eqnarray*}
where for all $i$
\begin{eqnarray*}
\chi_i,\varphi_i \in \mathcal{D}\Big(B(a_i,\frac{3r_a}{2})\Big),
\qquad \nabla \varphi_i=\nabla^{\bot} \chi_i\: \: \textrm{on }\:
B(a_i,r_a)
\end{eqnarray*}
and $\chi_i$ (hence $\varphi_i$) is affine on $B(a_i,r_a)$ with $
|\nabla \chi_i(a_i)|=|\nabla \varphi_i(a_i)|\leq 1$.

By definition of $r_a$ such functions $\chi$ and $\varphi$ always
exist, and we can moreover estimate their $L^{\infty}$ norms by
\begin{equation*}
\|D\varphi\|_{\infty},\|D\chi\|_{\infty}\leq \frac{C}{r_a}, \qquad
\|D^2\varphi\|_{\infty},\|D^2\chi\|_{\infty}\leq \frac{C}{r_a^2}.
\end{equation*}

We next establish a control of the remainder terms appearing in
Proposition \ref{evolution}.

\begin{lemme}
Assume that $\displaystyle \sup_{0<\eps <1} T_\eps=T_*$ is finite.
Then there exists a constant $C=C(r_a,R,K_1,T_\ast)$ such that
 \label{lemma : dissipation}
\begin{equation*}
\label{energydissipation} \int_0^{T_\eps} \intr \frac{|\dt
u_\eps|^2}{\leps^2} \, ds \leq \frac{C}{\leps}
\end{equation*}
and for all $\chi, \varphi \in \mathcal{H}(a_i^0)$
\begin{equation*}
\label{reste} \left| \int_0 ^{T_\eps} \intr (\nabla ^{\bot} \chi
-\nabla \varphi)\cdot \frac{\dt u_\eps \cdot \nabla
u_\eps}{\leps}\,ds \right|\leq \frac{ C}{\leps ^{\frac{1}{2}}}.
\end{equation*}
\end{lemme}
\begin{proof}
 In order to prove the first inequality, we use Theorem \ref{cauchyproblem} and obtain
\begin{equation*}
\begin{split}
\frac{\delta}{\leps} \int_0^{T_\eps} \intr |\dt
u_\eps|^2=&\er\left(u_\eps^0\right)-\er\left(u_\eps(T_\eps)\right)\\
\leq &K_1+ \er\left(\uest(a_i^0,d_i)\right)-\er\left(u_\eps
(T_\eps)\right).
\end{split}
\end{equation*}
 Since $n(u_\eps(T_\eps))\leq
n_1$ we have by Lemma \ref{linkenergies}
\begin{equation*}
\begin{split}
\er\left(\uest(a_i^0,di)\right)-&\er\left(u_\eps(T_\eps)\right)
\\&\leq \int_{B(R')} e_\eps\left(\uest(a_i^0,d_i)\right)-\int_{B(R')}
e_\eps\left(u_\eps(T_\eps)\right) +\frac{C}{R'}
\end{split}
\end{equation*}
which is bounded in view of \eqref{upperbound3}. It then suffices to
divide all terms by $\leps$.

\medskip

For the second assertion, we set  $\xi=\nabla^{\bot} \chi-\nabla
\varphi$ which has compact support in $A=\cup A_i$, where $A_i=
B(a_i^0,\frac{3r_a}{2})\setminus B(a_i^0,r_a)$, and we apply
Cauchy-Schwarz inequality. We obtain
\begin{equation*}
\begin{split}
\left( \int_0 ^{T_\eps} \intr (\nabla ^{\bot} \chi -\nabla
\varphi)\cdot \frac{\dt u_\eps \cdot \nabla u_\eps}{\leps}\right)^2&
\\
\leq \left( \int_0 ^{T_\eps} \intr \frac {|\dt
u_\eps|^2}{\leps^2}\right)&\cdot \left( \int_0 ^{T_\eps} \int_{A}
|\nabla u_\eps|^2 |\xi|^2\right).
\end{split}
\end{equation*}
Since $A \subset \Omega_{R',r_a}$, we infer from \eqref{upperbound2}
\begin{equation*}
  \int_0 ^{T_\eps} \int_{A} |\nabla u_\eps|^2 |\xi|^2\leq  \|\xi\|_{\infty}^2 \int_0^{T_\eps} \int_{A}
|\nabla u_\eps|^2 \leq CT_*\|\xi\|_{\infty}^2,
\end{equation*}
and the conclusion finally follows from the first part of the proof.
\end{proof}

We may now establish the following

\begin{lemme}
\label{liminfteps} There exists $T=T(r_a,R_a,R,K_1)>0$ such that
\begin{eqnarray*}
\liminf_{\eps \to 0} T_\eps \geq T.
\end{eqnarray*}
\end{lemme}

\begin{proof}

������������������������������
The first step consists in showing that for $(\chi,\varphi)\in
\mathcal{H}(a_i^0)$, for $s,t\in[0,T_\eps]$ and $i=1,\ldots,l$ we
have
\begin{equation}
\label{step1}
\begin{split}
\big| \langle \chi_i,Ju_\eps(t)-Ju_\eps(s) \rangle + \delta \langle
\varphi_i,\mu_\eps(t)&-\mu_\eps(s) \rangle\big| \\
&\leq C |t-s| +\frac{C}{\leps^{\frac{1}{2}}}.
\end{split}
\end{equation}
Indeed, we fix $i$ and we invoke Proposition \ref{evolution} for
$u\equiv u_\eps$ and the choice of
 test functions $(\chi_i,\varphi_i)$. Integrating \eqref{formuleevolution} on $[s,t]$ yields
\begin{equation*}
\begin{split}
| \langle \chi_i,Ju_\eps(t)-&Ju_\eps(s) \rangle + \delta \langle
\varphi_i,\mu_\eps(t)-\mu_\eps(s) \rangle| \leq 2\int_s^t \int
\Big|\textrm{Im} \left(\omega(u_\eps)\frac{\partial^2
\chi_i}{\partial
\overline{z}^2}\right)\Big|\\
&+\int_s^t \int \left|\frac{|\dt u_\eps|^2}{\leps^2}\varphi_i
+(\nabla^{\bot} \chi_i -\nabla \varphi_i)\cdot \frac{\dt u_\eps
\cdot \nabla u_\eps}{\leps}\right|,
\end{split}
\end{equation*}
where $\frac{\partial^2 \chi_i}{\partial \overline{z}^2}$ has
support in $C_i \subset \Omega_{R',r_a}$, and it finally suffices to
use \eqref{upperboundomega} and Lemma \ref{lemma : dissipation}.

\medskip

In a second step, we take advantage of the equality $$
\|Ju_{\eps}(T_\eps)-\pi \ssum d_i \delta_{a_i^0}\|\equiv \eta_0.$$
We set
\begin{eqnarray*}
\nu_{i,\eps}=d_i
\frac{b_i^\eps(T_\eps)-a_i^0}{|b_i^\eps(T_\eps)-a_i^0|},\qquad
i=1,\ldots,l
\end{eqnarray*}
and we define $\chi_{i,\eps}$, $\varphi_{i,\eps}$ so that for $x\in
B(a_i^0,r_a)$,
\begin{eqnarray*} \chi_{i,\eps}(x)=\nu_{i,\eps}\cdot x,\qquad
\varphi_{i,\eps}(x)= \nu^{\bot}_{i,\eps}\cdot x,
\end{eqnarray*}
and we require additionally that $\chi=\sum \chi_{i,\eps}$ and
$\varphi=\sum \varphi_{i,\eps}$ belong to $\mathcal{H}(a_i^0)$; we
can moreover choose $\varphi_{i,\eps}$ and $\chi_{i,\eps}$ so that
their norms in $C^2(B(R))$ remain bounded uniformly in $\eps$. As
$b_i^\eps(T_\eps) \in B(a_i^0,r_a/2)$, we have
\begin{eqnarray*}
|d_i| |b_i^\eps(T_\eps)-a_i^0|=d_i\chi
(b_i^\eps(T_\eps)-a_i^0)+\delta \varphi(b_i^\eps(T_\eps)-a_i^0),
\end{eqnarray*}
so that
\begin{eqnarray*}
\| \pi \sum_{i=1}^l d_i (\delta_{b_i^{\eps}(T_\eps)}-\delta_{\aio})
\|=\langle \pi \sum_{i=1}^l d_i
(\delta_{b_i^{\eps}(T_\eps)}-\delta_{\aio}) , \chi\rangle +\delta
\langle \pi \sum_{i=1}^l (
\delta_{b_i^{\eps}(T_\eps)}-\delta_{\aio}), \varphi\rangle.
\end{eqnarray*}
On the other hand, we have
\begin{eqnarray*}
\|Ju_\eps(T_\eps)-\pi \ssum d_i \delta_{\aio}\| \leq \|
Ju_\eps(T_\eps)-\pi \ssum {d_i \delta}_{b_i^{\eps}(T_\eps)} \|+\|\pi
\ssum d_i( \delta_{b_i^{\eps}(T_\eps)}-\delta_{\aio})\|.
\end{eqnarray*}
The second term in the right-hand side may be rewritten as
\begin{eqnarray*}
\langle \pi \ssum d_i (\delta_{b_i^{\eps}(T_\eps)}-\delta_{\aio}) ,
\chi\rangle +\delta \langle \pi \ssum (
\delta_{b_i^{\eps}(T_\eps)}-\delta_{\aio}), \varphi\rangle=A+B+C,
\end{eqnarray*}
where
\begin{eqnarray*}
A&=&\langle \pi \ssum d_i
\delta_{b_i^{\eps}(T_\eps)}-Ju_\eps(T_\eps),\chi\rangle + \delta
\langle \pi \ssum
\delta_{b_i^{\eps}(T_\eps)}-\mu_\eps(T_\eps),\varphi\rangle \\
&\leq&C\left(
 \|
Ju_\eps(T_\eps)-\ssum d_i \delta_{b_i^\eps(T_\eps)}\|+\delta  \|
\mu_\eps(T_\eps)-\ssum \delta_{b_i^\eps(T_\eps)}\|\right),
\end{eqnarray*}
$B$ is given by
\begin{eqnarray*}
B=\langle Ju_\eps(T_\eps)-Ju_\eps(0),\chi\rangle+ \delta \langle
\mu_\eps(T_\eps)-\mu_\eps(0),\varphi\rangle
\end{eqnarray*}
and finally
\begin{eqnarray*}
C&=&\langle Ju_\eps^0 - \pi \ssum d_i \delta_{\aio},\chi\rangle +
\delta \langle \mu_\eps(u_\eps^0)-\pi \ssum \delta_{\aio},\varphi
\rangle \ \\
&\leq& C\left(
 \|
Ju_\eps^0-\ssum d_i \delta_{\aio}\|+\delta  \|
\mu_\eps(u_\eps^0)-\ssum \delta_{\aio}\|\right).
\end{eqnarray*}
In view of the bound provided by \eqref{step1} for $B$, estimates
\eqref{concentrationjacobien}- \eqref{concentrationenergie} and the
fact that $\Sigma_\eps^0(s)\leq K_1$ for $0\leq s\leq T_\eps$, this
implies
\begin{equation*}
\begin{split}
\eta_0 =\|Ju_\eps(T_\eps)-\pi \ssum d_i \delta_{\aio}\|\leq C( \eps
\leps + \leps^{-1}+\leps^{-\frac{1}{2}})+C T_\eps,
\end{split}
\end{equation*}
and letting $\eps \to 0$ yields the conclusion. Lemma
\ref{liminfteps} is proved.
\end{proof}

\noindent \textbf{{Proof of Theorem \ref{theoreme1bis} completed.}}
\\  We consider $t,s \in [0,T]$. Arguing as in the proof of Lemma \ref{liminfteps} (with $T_\eps$
and $0$ replaced by $t$ and $s$), we find that for all
$\chi,\varphi$ belonging to $\mathcal{H}(\aio)$
\begin{equation*}
\begin{split}
\Big| \ssum d_i \Big[ \chi &(b_i^\eps(t))- \chi
(b_i^{\eps}(s))\Big]+
\delta \Big[\varphi (b_i^{\eps}(t))-\varphi (b_i^{\eps}(s))\Big]\Big| \\
&\leq C \sup_{\tau \in [0,T]}\Big( \| Ju_\eps(\tau)-\ssum d_i
\delta_{b_i^\eps(\tau)}\|+\delta  \| \mu_\eps(\tau)-\ssum
\delta_{b_i^\eps(\tau)}\|\Big)\\
&+\big|\langle Ju_\eps(t)-Ju_\eps(s),\chi\rangle+ \delta \langle
\mu_\eps(t)-\mu_\eps(s),\varphi\rangle\big|,
\end{split}
\end{equation*}
which is bounded by $ o_\eps(1)+ c|t-s|$ by
\eqref{concentrationjacobien}-\eqref{concentrationenergie} and
\eqref{step1}. Considering successively $\chi(x)=e_1\cdot x$ and
$\chi(x)=e_2\cdot x$ on each $B(\aio,r_a)$, we obtain
\begin{equation}
\label{lipschitz} |b_i^\eps(t)-b_i^\eps(s)| \leq c |t-s|+o_\eps(1).
\end{equation}

 Next, using that $b_i^\eps \in
B(\aio,r_a)$ and a standard diagonal argument, we may construct a
sequence $(\eps_k)\to 0$ and paths $b_i(t)$ such that
$b_i^{\eps_k}(t)$ converges to $b_i(t)$ for all $t\in \mathbb{Q}\cap
[0,T]$. We infer then from
\eqref{concentrationjacobien}-\eqref{concentrationenergie} that the
convergence statements
\eqref{concentrationjacobien2}-\eqref{concentrationenergie2} in
Theorem \ref{theoreme1bis} hold for these times. Moreover, in view
of \eqref{lipschitz} these paths are Lipschitz on $[0,T]\cap
\mathbb{Q}$, so that they can be extended in a unique way to
Lipschitz paths (still denoted by $b_i(t)$) on the whole of $[0,T]$.
We can finally establish that the convergence
\eqref{concentrationjacobien2}-\eqref{concentrationenergie2} holds
uniformly with respect to $t\in [0,T]$ by using again
\eqref{lipschitz} and
\eqref{concentrationjacobien}-\eqref{concentrationenergie}.

\medskip

Finally, we already know from \eqref{eb} that estimate
\eqref{energieanneau} holds for the full family
$(u_\eps)_{\eps<\eps_{\Lambda}}$. In order to show
\eqref{energieresteO}, we recall first the uniform bound $\er
\left(u_\eps(t)\right)-\er \left(\uest (\aio,d_i)\right) \leq K_1. $
On the other hand, Corollary \ref{renormalizedenergy} gives
\begin{eqnarray*}
\er \left(\uest (\aio,d_i)\right)-\er \left(\uest
(b_i(t),d_i)\right)= W(\aio,d_i)-W(b_i(t),d_i)\leq C,
\end{eqnarray*}
since the $b_i's$ are continuous and remain separated on $[0,T]$.
This yields the bound \eqref{energieresteO} and concludes the proof
of Theorem \ref{theoreme1bis}.
\end{proof}

As mentioned in the beginning of the proof of Theorem
\ref{theoreme1bis},
 the convergence of the initial data in \eqref{WWP1} actually holds on every large ball
 $B(L),L=2^n \geq R$, so
  that we find the same conclusions when replacing $R$ by $L$.

\begin{lemme}[\cite{BJS}, Lemma 7.3]
\label{largerballs} There exists a subsequence, still denoted by
$\eps_k$, such that for all $L\geq 2^{n_1}$,
\begin{eqnarray*}
\eta_k:=\sup_{[0,T]} \| Ju_{\eps_k}(t) - \pi \ssum d_i
\delta_{b_i(t)}\|_{W_0^{1,\infty}(B(L))^*} \to 0,\qquad k\to
+\infty.
\end{eqnarray*}
\end{lemme}

For $t\in [0,T]$ and sufficiently large $k\in \mathbb{N}$, we may
therefore apply
 Theorem \ref{coercivitytheorem} to $u_{\eps_k}(t)$ with respect to the configuration $(b_i(t),d_i)$
  and with the choice $R_0=L=2^n$ for each $n\geq n_1$. We are led
  to
  introduce the excess energy at time $t$ with respect to the configuration
$(b_i(t),d_i)$ by
\begin{equation*}
\Sigma_{\eps_k}(t)=\mathcal{E}_{\eps_k,U_d}\left(u_{\eps_k}(t)\right)-\mathcal{E}_{\eps_k,U_d}\left(u_{\eps_k}^\ast(b_i(t),d_i)\right),
\end{equation*}
which is uniformly bounded on $[0,T]$ in view of
\eqref{energieresteO}. Letting first $k$, then $n$ tend to
$+\infty$, we can get rid of the dependance on $R$ in
\eqref{coercivity}.

\begin{lemme}
\label{coercivitylemma}
 For all $r \leq r_a/2$ and $K\geq 2^{n_1}$, we
have for sufficiently large $k$ and $t,t_1,t_2 \in [0,T]$
\begin{equation*}
\begin{split}
\int_{B(K)\setminus \cup
B(b_i(t),r)}e_{\eps_k}(|u_{\eps_k}(t)|)+\frac{1}{8}
\Big|\frac{j(u_{\eps_k}(t))}{|u_{\eps_k}(t)|}
&-j(u^*(b_i(t),d_i))\Big|^2 \\
&\leq \Sigma_{{\eps}_k}(t)+C(\eps_k,\eta_k,\frac{1}{K}).
\end{split}
\end{equation*}
Therefore, we have as $k\to + \infty$
\begin{equation*}
\begin{split}
\limsup_{k\to +\infty}\int_{t_1}^{t_2} \int_{B(K)\setminus \cup
B(b_i(t),r)}e_{\eps_k}(|u_{\eps_k}(t)|)&+\frac{1}{8}
\Big|\frac{j(u_{\eps_k})(t)}{|u_{\eps_k}(t)|}
-j(u^*(b_i(t),d_i))\Big|^2 \\
&\leq \limsup_{k \to + \infty}\int_{t_1}^{t_2}\Sigma_{{\eps}_k}(t).
\end{split}
\end{equation*}
\end{lemme}
Consequently, it appears that the distance between $u_{\eps_k}(t)$
and $\ust(b_i(t),d_i)$ may be asymptotically entirely controlled by
$\limsup \Sigma_{\eps_k}(t)$.

\medskip

We now define the trajectory set
\begin{equation*}
\mathcal{T}=\{ (t,b_i(t)),\: t\in [0,T],\: i=1,\ldots,l\}
\end{equation*}
and
\begin{equation*}
\mathcal{G}=[0,T]\times \R \setminus \mathcal{T}.
\end{equation*}
Thanks to the uniform bounds in $L_\loc^2(\mathcal{G})$ provided by
Lemma \ref{coercivitylemma}, we establish the following
\begin{prop}
\label{cvaway} There exists a subsequence, still denoted $\eps_k$,
such that
\begin{equation*}
\frac{j(u_{\eps_k})}{|u_{\eps_k}|} \rightharpoonup j(\ust
(b_i(\cdot),d_i))
\end{equation*}
weakly in $L_{\loc}^2(\mathcal{G})$ as $k\to + \infty$.
\end{prop}
\begin{proof} Let $B$ be any bounded subset of $\R$.
First, we observe that according to Lemma \ref{largerballs}
\begin{equation}
\label{cvrot} \textrm{curl}\big(j(u_{\eps_k})\Big)=2 Ju_{\eps_k} \to
2\pi \sum_{i=1}^l d_i
\delta_{b_i(\cdot)}=\textrm{curl}\Big(j(\ust(b_i(\cdot),d_i))\Big)
\end{equation}
in $\mathcal{D}'([0,T]\times B)$.

\medskip

On the other hand, we have
\begin{equation}
\label{cvdiv} \textrm{div}\Big(j(u_{\eps_k})\big)\to
0=\textrm{div}\big(j\ust((b_i(\cdot),d_i))\Big)
\end{equation}
in $\mathcal{D}'([0,T]\times B)$.

\medskip

Indeed, since $u_{\eps_k}$ solves \CGL we obtain by considering
 the exterior product
\begin{equation*}
k_{\eps_k} u_{\eps_k}\times \dt u_{\eps_k}+u_{\eps_k}\cdot \dt
u_{\eps_k}=u_{\eps_k} \times \Delta u_{\eps_k}=\textrm{div}\big
(j(u_{\eps_k})\big),
\end{equation*}
so we are led to
\begin{equation}
\label{div}
 \textrm{div} (ju_{\eps_k})=k_{\eps_k} u_{\eps_k}\times
\dt u_{\eps_k}+ \frac{1}{2} \eps_k \frac{d}{dt}\left(
\frac{|u_{\eps_k}|^2-1}{\eps_k}\right).
\end{equation}

Now, applying Lemma \ref{linkenergies} to $u_{\eps_k}$, we find
\begin{equation}
\label{ineq:energy} \sup_{[0,T]}E_{\eps_k}(u_{\eps_k}(t),B) \leq \pi
l \leps + \Sigma_{\eps_k}(t) +C\leq \pi l \leps +C,
\end{equation}
where the second inequality is itself a consequence of
\eqref{energieresteO}. This implies first that $|u_{\eps_k}|\to 1$
in $L^2([0,T]\times B)$. Moreover, we infer that the second term in
the r.h.s of \eqref{div} converges to zero in the distribution sense
on $[0,T]\times B$. For the first one, it suffices to use
Cauchy-Schwarz inequality combined with the $L^2$ bound provided by
Lemma \ref{energydissipation} and the already mentioned uniform
bounds of $|u_{\eps_k}|$ in $L^2_{\loc}$.

We then infer from Lemma \ref{largerballs} and \eqref{ineq:energy}
that $j(u_{\eps_k})$ is uniformly bounded in $L_{\loc}^p([0,T]\times
\R)$ for all $p<2$. This is e.g. a consequence of Theorem 3.2.1 in
\cite{CJ2} and the remarks that follow. We deduce from \eqref{cvrot}
and \eqref{cvdiv} that up to a subsequence, we have
\begin{equation}
\label{convergence} j(u_{\eps_k})\rightharpoonup
j_1=j(\ust(b_i(\cdot),d_i))+H
\end{equation}
weakly in $L_{\loc}^p([0,T]\times \R)$, where $H$ is harmonic in $x$
on $[0,T]\times \R$.

On the other hand, it follows from the first part of Lemma
\ref{coercivitylemma} that there exists $j_2$ such that, taking
subsequences if necessary, $j(u_{\eps_k})/|u_{\eps_k}|
\rightharpoonup j_2$ weakly in $L_{\loc}^2(\mathcal{G})$.

Taking into account the strong convergence $|u_{\eps_k}|\to 1$ in
$L_{\loc}^2([0,T]\times \R)$, we obtain $j_1=j_2 \in
L_{\loc}^2(\mathcal{G})$. The second part of Lemma
\ref{coercivitylemma} combined with \eqref{convergence} then yields
\begin{equation*}
\|H\|_{L_{\loc}^2(\mathcal{G})}\leq \liminf_{k \to +\infty}
\|\frac{j(u_{\eps_k})}{|u_{\eps_k}|}-j(u^*(b_i,d_i))\|_{L_{\loc}^2(\mathcal{G})}\leq
CT,
\end{equation*}
where $C$ depends only on $K_1$, $R$ and $r_a$, so finally
$\|H\|_{L^2([0,T]\times \R)} \leq CT$. Since $H$ is harmonic in $x$,
we find that $H(t,\cdot)$ is bounded on $\R$ for almost every $t$
and therefore is identically zero. We end up with
$j_1=j_2=j(\ust(b_i(\cdot),d_i))$ in $\mathcal{G}$, and the
conclusion follows.
\end{proof}


\section{Proof of Theorem \ref{theorem1}}

In this section, we present the proof of Theorem \ref{theorem1}. We
let $\{b_i(t)\}$ be the $l$ Lipschitz paths on $[0,T]$ provided by
Theorem \ref{theoreme1bis} and $\{a_i(t)\}$ be the unique maximal
solution defined on $I=[0,T^\ast)$ to \eqref{systemepointsvortex}
with initial conditions $\aio$. Our aim is to show that
$a_i(t)\equiv b_i(t)$ on $I$. We will first prove that this holds on
$[0,T]$. By Rademacher's Theorem, the time derivatives
$\dot{b_i}(t)$ exist and are bounded almost everywhere on $[0,T]$.
Without loss of generality, we may assume $T<T^\ast$, so that
\begin{equation}
\label{bornelip}
\begin{split}
|\dot{a_i}(t)|\leq C,\:\: |\dot{b_i}(t)|\leq C,\qquad \textrm{a.e.
on } [0,T].
\end{split}
\end{equation}
Moreover, we may assume, decreasing possibly $T$, that
$|a_i(t)-b_i(t)|\leq r_a/2$ for all $i$. Hence, the trajectories
$a_i(t)$ remain in $B(\aio,r_a)$ on $[0,T]$. We introduce
\begin{eqnarray*}
h(t)=\sum_{i=1}^l \int_0 ^t |\dot{a_i}(s)-\dot{b_i}(s)|\,ds,\qquad
\sigma(t)=\sum_{i=1}^l |a_i(t)-b_i(t)|,
\end{eqnarray*}
then $h$ is Lipschitz on $[0,T]$ and for almost every $t\in [0,T]$
we have $h'(t)=\ssum |\dot{a_i}(t)-\dot{b_i}(t)|$. Note that since
$\sigma$ is absolutely continuous and $\sigma(0)=0$, we have for all
$t\in [0,T]$
\begin{equation*}
\sigma(t)=\int_0^t \sigma'(s)\,ds \leq h(t)
\end{equation*}
therefore it suffices to show that $h$ is identically zero on
$[0,T]$. This will be done by mean of Gronwall's Lemma.

\begin{lemme}
\label{dissipationexcesenergie} For all $t_1,t_2,t\in [0,T]$, we
have
\begin{eqnarray*}
\limsup_{k \to +\infty} \Sigma_{\eps_k}(t) \leq C  h(t)
\end{eqnarray*}
and
\begin{eqnarray*}
\limsup_{k \to +\infty} \int_{t_1}^{t_2} \Sigma_{\eps_k}(s)\,ds \leq
C \int_{t_1}^{t_2} h(s)\,ds,
\end{eqnarray*}
where $C$ only depends on $r_a, K_0, R_a$.
\end{lemme}

\begin{proof}
\newcommand{\Si}{\Sigma_{\eps_k}}
\newcommand{\ek}{\eps_k}
\newcommand{\lk}{|\log \ek|}
 For $t\in [0,T]$, we decompose $\Si(t)$ as
\begin{equation*}
\begin{split}
\Si(t)=\mathcal{E}_{\ek,[U_d]}\left(u_{\ek}(t)\right)&-\mathcal{E}_{\ek,[U_d]}(u_{\ek}^0)+\Si(0)
\\&+\mathcal{E}_{\ek,[U_d]}\left(u_{\ek}^*(\aio,d_i)\right)-\mathcal{E}_{\ek,[U_d]}\left(u_{\ek}^*(b_i(t),d_i)\right).
\end{split}
\end{equation*}
Appealing to Corollary \ref{renormalizedenergy} and Theorem
\ref{cauchyproblem}, we obtain
\begin{equation*}
\Si(t)=-\delta \int_0 ^t\intr \frac{|\dt
u_{\ek}|^2}{{\lk}}+\Si(0)+W(\aio,d_i)-W(b_i(t),d_i)+o_{\eps_k}(1).
\end{equation*}
Using that $W$ is Lipschitz away from zero, we estimate the last
term as follows
\begin{equation*}
\begin{split}
W(\aio,d_i)-W(b_i(t),d_i)&=W(\aio,d_i)-W(a_i(t),d_i)+
W(a_i(t),d_i)-W(b_i(t),d_i)\\
&\leq -\int_0^t \sum_{i=1}^l \dot{a_i}(s)\cdot \nabla_{a_i} W(s)\,
ds + C \sigma(t).
\end{split}
\end{equation*}
Since the $a_i$ solve the Cauchy problem
\eqref{systemepointsvortex}, an explicit computation gives
\begin{eqnarray*}
\dot{a_i}(s)\cdot \nabla_{a_i}W(s) =\frac{\delta}{\pi} C_i d_i
|\nabla_{a_i}W|^2=-\delta \pi |\dot{a_i}(s)|^2,
\end{eqnarray*}
so that
\begin{eqnarray*}
\Si(t)\leq \Si(0)+\delta \pi \int_0^t \sum_{i=1}^l |\dot{a_i}(s)|^2
\,ds -\delta \int_0 ^t\intr \frac{|\dt u_{\ek}|^2}{{\lk}}+
C\sigma(t)+o_{\eps_k}(1).
\end{eqnarray*}
We handle next the energy dissipation in the right-hand side. In
view of Lemma \ref{lemma : dissipation}, we have $\int_{[0,T]\times
\R} |\dt u_{\eps_k}|^2 \leq C |\log\eps_k|$, while
$E_{\ek}(u_{\ek},B(R')) \leq \pi l |\log\eps_k|+C$. Applying
Corollary 7 in \cite{SS} to $(u_{\ek})$, we obtain
\begin{eqnarray}
\label{sandserf} \liminf_{k\to + \infty} \int_0^t \intr \frac{|\dt
u_{\ek}|^2}{{\lk}}\geq \pi \sum_{i=1}^l \int_0 ^t |\dot{b_i}(t)|^2
\,ds.
\end{eqnarray}

Now, we have thanks to \eqref{bornelip}
\begin{equation*}
\sum_{i=1}^l \int_0^t
\big(|\dot{a_i}(s)|^2-|\dot{b_i}(s)|^2\Big)\leq C \sum_{i=1}^l
\int_0^t |\dot{a_i}(s)-\dot{b_i}(s)|\,ds=Ch(t),
\end{equation*}
whereas $\Si(0)\to 0$ by assumption, hence we get
\begin{eqnarray*}
\limsup_{k\to + \infty}\Si(t)\leq C\left(\sigma(t)+h(t)\right).
\end{eqnarray*}
Applying Fatou's Lemma in \eqref{sandserf} finally also provides the
corresponding integral version, and lastly, it suffices to use that
$\sigma \leq h$.
\end{proof}
As suggested in the introduction, the map $\ust(a_i(t),d_i)$ solves
the evolution formula given in Proposition \ref{evolution} in the
asymptotics $\eps \to 0$.
\begin{lemme}
\label{evolutionasymptotique} We have for $t\in[0,T]$ and
$\chi,\varphi \in \mathcal{H}(a_i^0)$
\begin{equation*}
\begin{split}
\pi \frac{d}{dt} \sum_{i=1}^l d_i &\chi \left(a_i(t)\right)+ \delta
\varphi \left(a_i(t)\right)=2 \intr\mathrm {Im}  \left(\omega
\left(\ust(a_i(t),d_i)\right) \frac{\partial^2 \chi }{\partial
\overline{z}^2}\right).
\end{split}
\end{equation*}
\end{lemme}
\begin{proof}
We use the following formula proved in \cite{BOS1}, valid for any
configuration $(a_i,d_i)$ and any test function $\chi$ which is
affine near the point vortices .
\begin{eqnarray*}
2 \intr \textrm {Im} \left(\omega (\ust(a_i(t),d_i))
\frac{\partial^2 \chi }{\partial \overline{z}^2}\right)=-\pi
\sum_{i\neq j} d_i d_j
\frac{(a_i(t)-a_j(t))^{\bot}}{|a_i(t)-a_j(t)|^2}\cdot \nabla
\chi(a_i(t)).
\end{eqnarray*}
On the other hand, we compute
\begin{eqnarray*}
\frac{d}{dt}\left( \sum_{i=1}^l d_i \chi (a_i)+ \delta \varphi
(a_i)\right)&=&\sum_{i=1}^l \Big( d_i \nabla \chi(a_i^0)\cdot
\dot{a_i}(t) +\delta \nabla \varphi(a_i^0)\cdot \dot{a_i}(t)\Big)
\\&=& \sum_{i=1}^l d_i \nabla \chi(a_i^0)\cdot \big(
\dot{a_i}(t)-\delta d_i \dot{a_i}^{\bot}(t)\big),
\end{eqnarray*}
where the second equality follows from the relation $ \nabla
\varphi(a_i^0)=\nabla^{\bot} \chi(a_i^0)$. Next, we deduce from
\eqref{systemepointsvortex}
\begin{equation*}
\pi \Big(\dot{a_i}(t)-\delta d_i
\dot{a_i}^{\bot}(t)\Big)=-C_i(1+\delta^2 d_i^2)\nabla_{a_i}^{\bot} W
=d_i \nabla_{a_i}^{\bot} W,
\end{equation*}
and we obtain
\begin{equation*}
\begin{split}
\pi \frac{d}{dt}\left( \sum_{i=1}^l d_i \chi (a_i)+ \delta \varphi
(a_i)\right)&=\sum_{i=1}^l \nabla \chi(a_i) \cdot
\nabla_{a_i}^{\bot} W \\
&= -\pi \sum_{i\neq j } d_i d_j \frac{(a_i-a_j)^{\bot}}{|a_i-a_j|^2}
\cdot \nabla \chi(a_i),
\end{split}
\end{equation*}
which yields the conclusion.
\end{proof}

\begin{lemme}
\label{lemma : estimate1} Set $A=\cup B(\aio,2r_a)\setminus
B(\aio,r_a)$ and let $t_1,t_2\in [0,T]$. Then for all $\varphi \in
\mathcal{D}(A)$, we have
\begin{eqnarray*}
\limsup_{k\to +\infty} \left| \int_{t_1}^{t_2}\int_A \left(
\omega\left(u_{\eps_k}(s)\right)-\omega\left(\ust(b_i(s),d_i)\right)\right)
\varphi \right| \leq C \|\varphi\|_{\infty} \int_{t_1}^{t_2}
h(s)\,ds.
\end{eqnarray*}
\end{lemme}

\begin{proof}
We apply the pointwise equality \eqref{omega} to $u\equiv
u_{\eps_k}(t)$ and $\ust\equiv \ust(b_i(t),d_i)$ for all $t$. Since
$|\ust(b_i(t),d_i)|=1$, this gives
\begin{eqnarray*}
\omega(u)-\omega(\ust)=\sum_{k,l=1}^2 \Big( a_{k,l}\partial_l |u|
\partial_k |u|+ b_{k,l}\Big[
\frac{j_k(u)}{|u|}\,
\frac{j_l(u)}{|u|}-j_k(\ust)j_l(\ust)\Big]\Big),
\end{eqnarray*}
where $a_{k,l},b_{k,l}\in \mathbb{C}$. We rewrite the terms
involving the components of $j$ as
\begin{equation*}
\begin{split}
\frac{j_k(u)}{|u|}\,
\frac{j_l(u)}{|u|}-j_k(\ust)j_l(\ust)&=\Big(\frac{j_k(u)}{|u|}-j_k(\ust)\Big)\,
\Big(\frac{j_l(u)}{|u|}-j_l(\ust)\Big)\\
+j_k(\ust)\Big(\frac{j_l(u)}{|u|}&-j_l(\ust)\Big)+j_l(\ust)\Big(\frac{j_k(u)}{|u|}-j_k(\ust)\Big).
\end{split}
\end{equation*}
We multiply the previous equality by $\varphi$, integrate on
$[t_1,t_2]\times A$ and
 let $k$ go to $+\infty$. Using the weak convergence in $L^2$ of
$j(u_{\eps_k})$ to $j(\ust(b_i(.),d_i))$ on $[0,T]\times A\subset
\mathcal{G}$ combined with the fact that $j\ust(b_i(.),d_i)$ is
bounded on this set, we deduce
\begin{equation*}
\begin{split}
\limsup_{k \to + \infty}\Big| \int_{t_1}^{t_2}\int_A &\Big(
\omega\left(u_{\eps_k}(s)\right)-\omega\left(\ust(b_i(s),d_i)\right)\Big)
\varphi
\Big|\\
&\leq c\|\varphi\|_{\infty} \limsup_{k \to +\infty} \int_{t_1}^{t_2}
\int_A \Big(|\nabla
|u_{\eps_k}||^2+\Big|\frac{ju_{\eps_k}}{|u_{\eps_k}|}-j\ust(b_i,d_i)\Big|^2\Big).
\end{split}
\end{equation*}
The conclusion finally follows from Lemmas \ref{coercivitylemma} and
\ref{dissipationexcesenergie}.
\end{proof}
We are now in position to complete the proof of Theorem
\ref{theorem1}. We consider arbitrary $\chi,\varphi$ belonging to
$\mathcal{H}(\aio)$, we fix $0\leq s\leq t\leq T$ and we integrate
the evolution formula \eqref{formuleevolution} on $[s,t]$. We obtain
\begin{eqnarray*}
\int_s^t \frac{d}{d\tau}\intr Ju_{\eps_k}(\tau)\chi +\delta \intr
\mu_{\eps_k}(\tau)\varphi= \int_s^t g_k^1(\tau)+\int_s^t
g_k^2(\tau),
\end{eqnarray*}
where
\begin{eqnarray*}
g_k^1(\tau)=-\delta \intr \frac{|\dt u_{\eps_k}|^2}{|\log
\eps_k|^2}+R_{\eps_k}(\tau,\chi,\varphi,u_{\eps_k})
\end{eqnarray*}
and
\begin{eqnarray*}
g_k^2(\tau)=2 \intr \textrm{Im} \Big( \omega(u_{\eps_k}(\tau))
\frac{\partial^2 \chi}{\partial \overline{z}^2}\Big),
\end{eqnarray*}
which we decompose as
\begin{equation*}
\begin{split}
 g_k^2&=2 \intr \textrm{Im} \left(\big[
\omega(u_{\eps_k})-\omega(\ust(b_i,d_i))\big]
\frac{\partial^2 \chi}{\partial \overline{z}^2}\right)\\
&+2 \intr \textrm{Im} \left( \big[
\omega(\ust(b_i,d_i))-\omega(\ust(a_i,d_i))\big]
\frac{\partial^2 \chi}{\partial \overline{z}^2}\right)\\
&+2 \intr \textrm{Im} \left( \omega(\ust(a_i,d_i)) \frac{\partial^2
\chi}{\partial \overline{z}^2}\right)=A_k(\tau)+B_k(\tau)+C_k(\tau).
\end{split}
\end{equation*}
We next substitute the formula given by Lemma
\ref{evolutionasymptotique} for $C_k$ in the previous equalities.
Setting
\begin{equation*}
f_{k,\chi,\varphi}(\tau)=\intr Ju_{\eps_k}(\tau)\chi +\delta \intr
\mu_{\eps_k}(\tau)\varphi-\pi \sum_{i=1}^l\Big( d_i
\chi(a_i(\tau))+\delta \varphi(a_i(\tau))\Big) ,
\end{equation*}
we obtain
\begin{equation*}
f_{k,\chi,\varphi}(t)-f_{k,\chi,\varphi}(s)=\int_s^t g_1^k+\int_s^t
A_k +\int_s^t B_k.
\end{equation*}
Lemma \ref{lemma : dissipation} with $T_\eps=T$ first gives
$|\int_s^t g_1^k(\tau)\,d\tau |\leq C|\log \eps_k| ^{-\frac{1}{2}}$
for all $k$. Moreover, it follows from Lemma \ref{lemma : estimate1}
and the fact that $\textrm{supp}\:\frac{\partial^2 \chi}{\partial
\overline{z}^2}\subset A$  that
\begin{equation*}
\limsup_{k \to +\infty} \Big|\int_s^t A_k(\tau)\, d \tau \Big|\leq C
\int_s^t h(\tau)\, d \tau.
\end{equation*}
Finally, we infer from the regularity of $\omega(\ust)$ away from
the vortices that
\begin{equation*}
 \int_s^t |B_k(\tau)|\,d \tau\leq C \int_s ^t \sigma(\tau)\, d
\tau\leq C\int_s^t h(\tau)\, d\tau.
\end{equation*}
Letting $k$ go to $+\infty$, we finally deduce from the convergence
statements in Theorem \ref{theoreme1bis} that for $0\leq s\leq t\leq
T$,
\begin{equation}
\label{r1} |f_{\chi,\varphi}(t)-f_{\chi,\varphi}(s)|\leq C \int_s^t
h(\tau)\,d\tau,
\end{equation}
where $f_{\chi,\varphi}$ is defined by
\begin{equation*}
f_{\chi,\varphi}=\pi \sum_{i=1}^l \Big[ d_i\big( \chi(b_i)-\chi(a_i)
\big) +\delta \big (\varphi(b_i)-\varphi(a_i)\big) \Big].
\end{equation*}
Here the constant $C$ depends only on $\chi$, $\varphi$ and the
initial conditions.

\medskip

We now fix a time $t\in [0,T]$ at which all the vortices $b_i$ have
a time derivative. Since the $a_i$ are $C^1$, it follows that
$f_{\chi,\varphi}$ is differentiable at $t$ with time derivative
given by
\begin{equation*}
f'_{\chi,\varphi}(t)= \pi \sum_{i=1}^l \Big( d_i
\nabla\chi(\aio)+\delta \nabla ^{\bot}\chi(\aio)\Big)\cdot \big(
\dot{b_i}(t)-\dot{a_i}(t)\big).
\end{equation*}
Dividing by $t-s$ in \eqref{r1} and letting $s\to t$ gives then
\begin{equation*}
\Big|\pi \sum_{i=1}^l \big( d_i \nabla\chi(\aio)+\delta \nabla
^{\bot}\chi(\aio)\big)\cdot \big(
\dot{b_i}(t)-\dot{a_i}(t)\big)\Big|\leq C\,h(t).
\end{equation*}
So, considering in particular $\chi,\varphi\in \mathcal{H}(\aio)$
such that $\chi$ and $\varphi$ vanish near each point $\aio$ except
for one, we obtain for all $i=1,\ldots,l$
\begin{equation*}
\Big|\pi \big(d_i \nabla\chi(\aio)+\delta \nabla
^{\bot}\chi(\aio)\big)\cdot \big(
\dot{b_i}(t)-\dot{a_i}(t)\big)\Big|\leq C\,h(t).
\end{equation*}
Choosing then successively $\chi(x)=x_1$ and $\chi(x)=x_2$ near
$\aio$ we end up with $|\dot{b_i}(t)-\dot{a_i}(t)|\leq C h(t)$, and
it follows by summation
\begin{equation*}
h'(t)\leq C h(t)\qquad \textrm{a.e. }\: t\in [0,T].
\end{equation*}
Since $h(0)=0$, this implies that $h=0$ on $[0,T]$, and hence
$\sigma=0$ on $[0,T]$.  Applying Lemma
\ref{dissipationexcesenergie}, we infer that $ \limsup_{k\to +
\infty} \Sigma_{\eps_k}(t)\leq 0$. Besides, Lemma \ref{linkenergies}
yields for all $L\geq 2^{n_1}$
\begin{equation*}
\begin{split}
\liminf_{k\to +\infty} \Sigma_{\eps_k}(t) &\geq \liminf_{k \to +
\infty} \Big( \int_{B(L)}
e_{\eps_k}\left(u_{\eps_k}(t)\right)-e_{\eps_k}\left( u_{\eps_k}
^\ast(a_i(t),d_i)\right) \Big) -\frac{C}{L}\\
&\geq -\frac{C}{L},
\end{split}
\end{equation*}
where the second inequality is a consequence of the convergence of
Jacobians on $B(L)$ stated in Lemma \ref{largerballs} (see
\cite{JS1,LX}). Letting $L$ tend to $+\infty$, we obtain
$\liminf_{k\to +\infty} \Sigma_{\eps_k}(t) \geq 0$, so we deduce
from \eqref{energieanneau} that  $(u_{\eps_k}(t))_{k\in \mathbb{N}}$
is well-prepared with respect to the configuration $(a_i(t),d_i)$.
By uniqueness of the limit, this finally holds for the full family
$(u_\eps(t))_{0<\eps<1}$ on $[0,T]$.

In conclusion, we observe that in our definition $T$ only depends on
$K_1$, $r_a$ and $\max(R,R_a+r_a)$, so that we can extend our
results to the whole of $[0,T^\ast)$ by repeating the previous
arguments.


\renewcommand{\thesection}{\textbf{\Alph{section}}}
\renewcommand{\thesubsection}{\thesection \textbf{\arabic{subsection}.}}

\renewcommand{\theequation}{\alph{equation}}

\newtheorem{propositionannexe}{\textbf{Proposition}}[section]
\newtheorem{lemmeannexe}{\textbf{Lemma}}[section]
\newcommand{\h}{h_{U_0}}
\newcommand{\f}{f_{U_0}}

\section*{Appendix}

\setcounter{section}{1} \setcounter{subsection}{1}
\setcounter{equation}{0}

We present here the proof of Theorem \ref{cauchyproblem}. We omit
the dependence on $\eps$ and rewrite \eqref{pert} in the following
way
\begin{equation}
\label{cauchy} \tag{CGL}
\begin{cases}
\dt w=(a+ib)\big(\Delta w+f_{U_0}(w)\big),\\
w(0)=w_0\in H^1(\R),
\end{cases}
\end{equation}
where
\begin{equation*}
f_{U_0}(w)=\Delta U_0+(U_0+w)(1-|U_0+w|^2),
\end{equation*}
$a$ is positive and $b\in \mathbb{R}$. We denote by $S=S(t,x)$ the
semi-group operator associated to the corresponding homogeneous
linear equation. Every solution to \eqref{cauchy} satisfies the
Duhamel formula
\begin{equation*}
w(t,\cdot)=S(t,\cdot)\ast w_0 +\int_0^t \big(S(t-s,\cdot)\ast
g_{U_0}(w(s),\cdot)\big)\,ds,
\end{equation*}
where $g_{U_0}=(a+ib) f_{U_0}$. The kernel $S$ is explicitly given
by
\begin{equation*}
S(t,x)=\frac{1}{4\pi (a+ib)t} \exp(\frac{-|x|^2}{4(a+ib)t}).
\end{equation*}
Since $a$ is positive, $S$ decays at infinity like the standard heat
Kernel. This will enable us to show that \eqref{cauchy} enjoys the
same smoothing properties as the parabolic Ginzburg-Landau equation.
In particular, we have for all $1\leq r\leq +\infty$ and for all
$t>0$
\begin{equation}
\label{decay1}
 \|S(t,\cdot)\|_{L^r(\R)}\leq
\frac{1}{t^{1-\frac{1}{r}}}
\end{equation}
 and concerning the space derivatives of $S(t)$,
\begin{equation}
\label{decay2} \|D^k S(t,\cdot)\|_{L^r(\R)} \leq
\frac{C(a,b)}{t^{\frac{|k|}{2}+1-\frac{1}{r}}}.
\end{equation}
We will often use Young's inequality that gives for $f\in L^p(\R)$
and $g\in L^q(\R)$ $\|f\ast g\|_{L^r(\R)}\leq \|f\|_{L^p(\R)}
\|g\|_{L^q(\R)}$, where $1+\frac{1}{r}=\frac{1}{p}+\frac{1}{q}$.
 We first state
local well-posedness for \eqref{cauchy}.

\begin{propositionannexe}
\label{localexistence}
 Let $w_0 \in H^1(\R)$. Then there exists a positive
time $T^\ast$ depending on $\|w_0\|_{H^1}$ and a unique solution
$w\in C^0([0,T^\ast),H^1(\R))$ to \eqref{cauchy}.
\end{propositionannexe}

\begin{proof}
We intend to apply the fixed point theorem to the map $\psi:w\in
H^1(\R)\mapsto \psi(w)$, where
\begin{equation*}
\psi(w)(t)= S(t)\ast w_0+\int_0^t S(t-s)\ast g_{U_0}(w(s))\,ds.
\end{equation*}
To this aim, we introduce $R=\|w_0\|_{H^1(\R)}$ and for $T>0$
\begin{equation*}
B(T,R)=\{w\in L^{\infty}([0,T],H^1(\R)) \:\textrm{s.t.}\:
\|w\|_{L^\infty(H^1)}\leq 3R\}.
\end{equation*}
We next show that we can choose $T=T(R)$ so that $\psi$ maps
$B(T(R),R)$ into itself and is a contraction on this ball.

For $T>0$, we let $w\in B(T,R)$ and expand $f_{U_0}(w)$. Using that
$H^1(\R)$ is continuously embedded in $L^p(\R)$ for all $2\leq
p<+\infty$ and the fact that $U_0$ belongs to $\mathcal{V}$, it can
be shown that \footnote{see Lemma 1 in \cite{BS}.}
\begin{equation}
\label{contraction1} \|\f\|_{L^\infty([0,T],L^2)}\leq C(U_0,R),
\end{equation}
and for $w_1,w_2\in B(T,R)$
\begin{equation}
\label{contraction2}
\|f_{U_0}(w_1)-f_{U_0}(w_2)\|_{L^\infty([0,T],L^2)}\leq
C(U_0,R)\|w_1-w_2\|_{L^\infty([0,T],H^1)}.
\end{equation}
We next apply Young's inequality to obtain
\begin{equation*}
\begin{split}
\|\psi(w)(t)\|_{H^1}&\leq
\|\psi(w)(t)\|_{L^2}+\|\nabla\psi(w)(t)\|_{L^2}\\
 &\leq2 \|S(t)\|_{L^1}\|w_0\|_{H^1}+\int_0^t
\|S(t-s)+\nabla S(t-s)\|_{L^1}\| g_{U_0}(s)\|_{L^2}\,ds\\
&\leq 2\|w_0\|_{H^1}+C \int_0^t \left(1+(t-s)^{-\frac{1}{2}}\right)
\| g_{U_0}(w(s))\|_{L^2}\,ds,
\end{split}
\end{equation*}
where the last inequality is a consequence of \eqref{decay1} and
\eqref{decay2} with the choice $r=1$.  This yields according to
\eqref{contraction1} and \eqref{contraction2}
\begin{equation*}
\sup_{t\in[0,T]}\|\psi(w)(t)\|_{H^1}\leq
2\|w_0\|_{H^1}+C(U_0,R)(T+\sqrt{T})
\end{equation*}
and similarly,
\begin{equation*}
\sup_{t\in[0,T]}\|\psi(w_1)(t)-\psi(w_2)(t)\|_{H^1}\leq
C'(U_0,R)(T+\sqrt{T})\sup_{t\in[0,T]}\|w_1(t)-w_2(t)\|_{H^1}.
\end{equation*}
The conclusion follows by choosing $T=T(R)$ sufficiently small so
that $2\|w_0\|_{H^1}+C(U_0,R)(T+\sqrt{T})\leq 3R$ and
$C'(U_0,R)(T+\sqrt{T})<1$.
\end{proof}
We next show additional regularity for a solution to \eqref{cauchy}.

\begin{lemmeannexe}
\label{reg} Let $w\in C^0([0,T],H^1(\R))$ be a solution to
\eqref{cauchy}. Then $w$ belongs to $L_{\loc}^1([0,T],H^2(\R))\cap
C^0((0,T],H^2(\R))$ and therefore to
$L_{\loc}^1([0,T],L^\infty(\R))$.
\end{lemmeannexe}

\begin{proof}
We first differentiate $\f(w)$ and use Lemma 2 in \cite{BS} which
states by mean of various Sobolev embeddings, H\"older and
Gagliardo-Nirenberg inequalities that
\begin{equation*}
\partial_i \f(w)=g_1(w)+g_2(w) \in
L^\infty([0,T],L^2(\R))+L^\infty([0,T],L^r(\R))
\end{equation*}
for all $1<r<2$. Moreover, we have
\begin{equation*}
\sup_{s\in [0,T]}\|g_1(w)(s)\|_{L^2(\R)}+\|g_2(w)(s)\|_{L^r(\R)}
\leq C(U_0,A(T),r),
\end{equation*}
where $ A(T)=\sup_{s\in[0,T]} \|w(s)\|_{H^1(\R)}$. Next,
differentiating twice Duhamel formula gives
\begin{equation*}
\partial_{ij}w(t)=\partial_j S(t)\ast \partial_i w_0+\int_0^t
\partial_jS(t-s)\ast \partial_i \f(s)\,ds,
\end{equation*}
so taking into account the decomposition $\partial_i
f_{U_0}=g_1+g_2$ we get
\begin{equation*}
\begin{split}
\|\partial_{ij}w(t)\|_{L^2}&\leq \|\nabla S(t)\|_{L^1} \|\nabla
w_0\|_{L^2}+\int_0^t \|\nabla
S(t-s)\|_{L^1}\|g_1(s)\|_{L^2}\,ds\\
&+\int_0^t \|\nabla S(t-s)\|_{L^\alpha}\|g_2(s)\|_{L^r}\,ds,
\end{split}
\end{equation*}
where $\alpha$ is chosen so that $
1+\frac{1}{2}=\frac{1}{\alpha}+\frac{1}{r}$. This finally yields in
view of \eqref{decay2}
\begin{equation*}
\begin{split}
\|\partial_{ij}w(t)\|_{L^2}\leq \frac{C}{t^{\frac{1}{2}}}\|
w_0\|_{H^1}+C(U_0,A(T),r)\int_0^t
\left((t-s)^{-\frac{1}{2}}+(t-s)^{-\frac{1}{2}-1+\frac{1}{\alpha}}\right)\,ds.
\end{split}
\end{equation*}
Since $\frac{1}{2}+1-\frac{1}{\alpha}=\frac{1}{r}<1$, we conclude
that the right-hand side is finite, so that $\partial_{ij} w(t) \in
L^2(\R)$.
\end{proof}
Lemma \ref{reg} enables to show that the renormalized energy is
non-increasing and to establish a control of the growth of
$\|w(t)\|_{H^1(\R)}$. For equation \eqref{cauchy}, this energy is
given by
\begin{equation*}
E_{U_0}(w)(t)=\int_{\R} \frac{|\nabla w|^2}{2}-\int_{\R} \Delta U_0
\cdot w+\int_{\R} \frac{(1-|U_0+w|^2)^2}{4}.
\end{equation*}
It is well-defined and continuous in time for $w\in C^0(H^1(\R))$.

\begin{lemmeannexe}
\label{control} Let $w\in C^0([0,T),H^1(\R))$ be a solution to
\eqref{cauchy}. Then for all $t\in(0,T)$ we have
\begin{equation*}
\frac{d}{dt} E_{U_0}(w)(t)\leq 0.
\end{equation*}
Moreover, there exists $C$ depending only on $\|w_0\|_{H^1}$ and
$U_0$ such that
\begin{equation}
\label{i3} \|w(t)\|_{H^1}\leq \|w_0\|_{H^1}\exp(Ct),\qquad \forall
t\in[0,T).
\end{equation}
\end{lemmeannexe}

\begin{proof}
We infer from equation \eqref{cauchy} and Lemma \ref{reg} that $\dt
w$ belongs to $L^\infty_{\loc}((0,T],L^2(\R))$, so that we may
compute
\begin{equation*}
\begin{split}
\frac{d}{dt} E_{U_0}(w(t))&=\int_{\R} \nabla w \cdot \nabla \dt
w-\Delta U_0\cdot \dt w -\dt w\cdot (U_0+w)(1-|U_0+w|^2)\\
&=-\int_{\R}\dt w\cdot ( \Delta w+\f(w))\\
&=-\int_{\R}\dt w\cdot (\frac{1}{a+ib}\dt
w)=\frac{-a}{a^2+b^2}\int_{\R} |\dt w|^2\leq 0.
\end{split}
\end{equation*}
We now turn to \eqref{i3}. We compute for $t\in(0,T)$
\begin{equation*}
\begin{split}
\frac{1}{2}\frac{d}{dt} \|w(t)\|_{L^2(\R)}^2&=\int_{\R} w\cdot \dt
w=\int_{\R} w\cdot [(a+ib)\Delta w]+\int_{\R} w \cdot
[(a+ib)\f(w)]\\
&=-a\int_{\R}|\nabla w|^2 +\int_{\R} w \cdot (a+ib)\Delta
U_0\\
&\hspace{1em}+\int_{\R} w \cdot [(a+ib)(U_0+w)(1-|U_0+w|^2)].
\end{split}
\end{equation*}
We then split the last term in the previous equality as
\begin{equation*}
\begin{split}
\int_{\R} w \cdot [(a+ib)(U_0+w)&(1-|U_0+w|^2)]=\int_{\R} w \cdot
[(a+ib)U_0(1-|U_0+w|^2)]\\
&+a\int_{\R} |w|^2(1-|U_0+w|^2).
\end{split}
\end{equation*}
The second term in the r.h.s. is clearly bounded by
$a\|w(t)\|_{L^2(\R)}$. Using Cauchy-Schwarz inequality for the first
one, we obtain
\begin{equation*}
\begin{split}
\int_{\R} w \cdot [(a+ib)(U_0+w)(1-|U_0+w|^2)] \leq
C(U_0)&\|w(t)\|_{L^2} V(t)^{\frac{1}{2}}+a\|w(t)\|_{L^2}^2,
\end{split}
\end{equation*}
where $V(t)=\int_{\R} (1-|U_0+w(t)|^2)^2$. We are led to
\begin{equation}
\label{i1} \frac{d}{dt} \|w(t)\|_{L^2(\R)}^2\leq
C(U_0)(\|w(t)\|_{L^2}^2+1+V(t)).
\end{equation}
On the other hand, Cauchy-Schwarz inequality gives
\begin{equation*}
E_{U_0}(w)(t)\geq \int_{\R} \frac{|\nabla
w|^2}{2}\,dx-C(U_0)\|w(t)\|_{L^2}+\frac{V(t)}{4},
\end{equation*}
which yields, since $E_{U_0}$ is non-increasing,
\begin{equation}
\label{i2} \frac{V(t)}{4}+\int_{\R} \frac{|\nabla w|^2}{2}\leq
E_{U_0}(w_0)+C(U_0)\|w(t)\|_{L^2}.
\end{equation}
We infer from \eqref{i1} and \eqref{i2}
\begin{equation*}
\|w(t)\|_{L^2}\leq (1+\|w_0\|_{H^1})\exp(Ct)
\end{equation*}
and finally deduce \eqref{i3} by using \eqref{i2} once more.
\end{proof}
Lemma \ref{control} provides global well-posedness for
\eqref{cauchy}.
\begin{propositionannexe}
Let $w_0\in H^1(\R)$. Then there exists a unique and global solution
$w\in C^0(\mathbb{R}_+,H^1(\R))$ to \eqref{cauchy}.
\end{propositionannexe}
\begin{proof}
Let $w\in C^0([0,T^\ast),H^1(\R))$ be the unique maximal solution
with initial condition $w_0$. If $T^\ast$ is finite, we have
according to \eqref{i3}
\begin{equation*}
\limsup_{t\to T^\ast} \|w(t)\|_{H^1(\R)}\leq
C(U_0,T^\ast,w_0)<+\infty,
\end{equation*}
so that we can extend $w$ to a solution $\overline{w}$ on
$[0,T^\ast+\delta]$. This yields a contradiction.
\end{proof}
We conclude this section with the following

\begin{propositionannexe}
\label{reg2} Let $w\in C^0(\mathbb{R}_+,H^1(\R))$ be the solution to
\eqref{cauchy}. Then we have $w\in
C^\infty(\mathbb{R}_+^\ast,C^\infty(\R))$.
\end{propositionannexe}
\begin{proof}

We proceed in several steps.

\medskip

\noindent \textbf{Step 1} Let $p\geq 2$ and $v\in H^p(\R)$. Then
$D^k\f(v)\in L^2(\R)+L^{\frac{4}{3}}(\R)$ for all $|k|\leq p$.

\medskip

\noindent \emph{Proof of Step 1.} We may assume in view of the proof
of Lemma \ref{reg} that $|k|\geq 2$. We decompose $\f(v)$ as
$\f(v)=\Delta U_0+\h(v)$, where
\begin{equation*}
 \h(v)=(U_0+v)(1-|U_0+v|^2).
\end{equation*}
Since $U_0\in \mathcal{V}$, it suffices to show that $D^k\h(v)\in
L^2(\R)+L^{\frac{4}{3}}(\R)$.  Applying Leibniz's formula to
$h_{U_0}(v)$, we obtain
\begin{equation*}
\begin{split}
D^k \h(v) &=\sum_{m\leq k}\binom{k}{m}D^{k-m} (U_0+v) D^{m}
(1-|U_0+v|^2)\\
 &=D^k(U_0+v)\\
 &-\sum_{\substack{m\leq k\\ n\leq m}}
\binom{k}{m}\binom{m}{n}D^{k-m} (U_0+v) D^{n} (U_0+v)\cdot
D^{m-n}(U_0+v).
\end{split}
\end{equation*}
Since $2\leq|k|\leq p$, $v\in H^p(\R)$ and $U_0\in \mathcal{V}$, we
clearly have $D^k(U_0+v)\in L^2(\R)$.

For the second term in the right-hand side, we write each product
inside the sum as
\begin{equation*}
D^a (U_0+v) D^b (U_0+v)\cdot D^c(U_0+v)
\end{equation*}
with $|a|+|b|+|c|=|k|\geq 2$, and we examine all cases. We observe
that $D^a (v+U_0)$ belongs to $H^1(\R)$ whenever $1\leq|a|\leq p-1$
and hence to $L^4(\R)$, whereas $D^a (v+U_0)$ belongs to $L^2(\R)$
for $2\leq |a|\leq p$. Since on the other hand $U_0+v\in L^\infty$,
we finally obtain
\begin{equation*}
D^a (U_0+v) D^b (U_0+v)\cdot D^c(U_0+v)\in
L^2(\R)+L^{\frac{4}{3}}(\R),
\end{equation*}
which yields the conclusion.

\medskip

We now turn to the regularity in space for a solution to
\eqref{cauchy}.

\medskip

\noindent \textbf{Step 2} Let $w\in C^0(\mathbb{R}_+,H^1(\R))$ be
the solution to \eqref{cauchy}. Then for all $p\geq 1$ we have $w\in
C^0(\mathbb{R}_+^\ast, H^{p}(\R))$.

\medskip

\noindent \emph{Proof of Step 2.} We proceed by induction on $p$.
The case $p=2$ has already been treated in Lemma \ref{reg}. Let us
thus assume that $w\in C^0(\mathbb{R}_+^\ast, H^{p}(\R))$ for some
$p\geq 2$. For $|k|\leq p+1$, we differentiate $w(t)$ and we find
\begin{equation*}
\begin{split}
D^k w(t)=D^k (S(t)\ast w_0)+D^k \int_0^t S(t-s)\ast g_{U_0}(s)\,ds
\end{split}
\end{equation*}
which we rewrite as
\begin{equation*}
\begin{split}
D^k w(t)=D^k S(t)& \ast w_0+\int_0^{t/2} (D^k S(t-s)) \ast
g_{U_0}(s)\,ds\\
&+\int_{t/2}^t D^{m} S(t-s)\ast D^{k-m} g_{U_0}(s)\,ds,
\end{split}
\end{equation*}
where $m$ is a multi-index so that $|m|=1$.

First, it follows from \eqref{decay2} that $t\mapsto D^k S(t) \ast
w_0\in C^0(\mathbb{R}_+^\ast, L^2(\R))$. Next, arguing that $g_{U_0}
\in C^0(\mathbb{R}_+, L^2(\R))$ and using \eqref{decay2} with $r=1$,
we find
\begin{equation*}
\begin{split}
\Big\| \int_0^{t/2} (D^k S(t-s)) \ast g_{U_0}(s)\,ds\Big\|_{L^2}\leq
C \int_0^{t/2} \frac{ds}{(t-s)^{\frac{|k|}{2}}}\leq
\frac{C}{t^{\frac{|k|}{2}-1}}.
\end{split}
\end{equation*}

On the other hand, since $|k-m|=|k|-1\leq p$ and since by assumption
$w(s)\in H^p(\R)$, Step 1 provides the decomposition
\begin{equation*}
D^{k-m}g_{U_0}(s)=d^1(s)+d^2(s)
\end{equation*}
where $d^1$ belongs to $C^0(\mathbb{R}_+^\ast, L^2(\R))$ and $d^2$
to $C^0(\mathbb{R}_+^\ast, L^{\frac{4}{3}}(\R))$. It follows from
\eqref{decay2} that
\begin{equation*}
\begin{split}
\Big\|\int_{t/2}^t D^m S(t-s)\ast D^{k-m}
g_{U_0}(s)\,ds\Big\|_{L^2}&\leq \int_{t/2}^t \|\nabla
S(t-s)\|_{L^1}\|d^1(s)\|_{L^2}\,ds\\
&\hspace{-6em}+\int_{t/2}^t \|\nabla
S(t-s)\|_{L^r}\|d^2(s)\|_{L^{\frac{4}{3}}}\,ds \\
&\leq C(t) \int_{t/2}^t
\left((t-s)^{-\frac{1}{2}}+(t-s)^{-\frac{1}{2}-1+\frac{1}{r}}\right)\,ds,
\end{split}
\end{equation*}
where $r$ satisfies $1+\frac{1}{2}=\frac{1}{r}+\frac{3}{4}$. The
last term is finite since $
\frac{1}{2}+1-\frac{1}{r}=\frac{3}{4}<1$, so we infer that $w \in
C^0(\mathbb{R}_+^\ast, H^{p+1}(\R))$, as we wanted.

\medskip

\noindent \textbf{Step 3} Let $w\in C^0(\mathbb{R}_+,H^1(\R))$ be
the solution to \eqref{cauchy}. Then we have  $w\in
C^k(\mathbb{R}_+^\ast, C^l(\R))$ for all $k,l\in \mathbb{N}$.

\medskip

\noindent \emph{Proof of Step 3.} For fixed $k,l\in \mathbb{N}$, we
show by induction on $0\leq j\leq k$ that $w\in
C^j(\mathbb{R}_+^\ast, C^{l+2k-2j}(\R))$.

\medskip
This holds for $j=0$ according to Step 2 and to Sobolev embeddings.
We assume next that $w\in C^j(\mathbb{R}_+^\ast, C^{l+2k-2j}(\R))$
for some $0\leq j\leq k-1$, and it follows that
\begin{equation*}
\Delta w,\: \f(w)\in C^j(\mathbb{R}_+^\ast,C^{l+2k-2j-2}(\R)).
\end{equation*}
So, going finally back to equation \eqref{cauchy}, we obtain
\begin{equation*}
 w \in
 C^{j+1}(\mathbb{R}_+^\ast,C^{l+2k-2j-2}(\R)).
\end{equation*}
This concludes the proof of Proposition \ref{reg2}.

\end{proof}

\subsection*{Acknowledgements} I warmly thank Didier Smets for his constant
support during the preparation of this work. I am also indebted to
Thierry Gallay and to Sylvia Serfaty for very helpful discussions.\\
 This work was partly supported by the grant JC05-51279 of
the Agence Nationale de la Recherche.



\end{document}